\documentclass{amsart}
\usepackage{amssymb,hyperref,mathtools}

\allowdisplaybreaks

 \newtheorem{thm}{Theorem}[section]
 \newtheorem{cor}[thm]{Corollary}
 \newtheorem{lem}[thm]{Lemma}
 \newtheorem{prop}[thm]{Proposition}
 \theoremstyle{definition}
 \newtheorem{defn}[thm]{Definition}
 \theoremstyle{remark}

 \numberwithin{equation}{section}

\newcommand {\B}    {\mathbb{B}}
\newcommand {\C}    {\mathbb{C}}

\newcommand {\R}    {\mathbb{R}}

\newcommand {\Z}    {\mathbb{Z}}
\newcommand {\T}    {\mathbb{T}}

\newcommand {\Aa}    {\mathcal{A}}

\newcommand {\Fa}    {\mathcal{F}}

\newcommand {\Ima}    {\hbox{\rm Im}\,}

\newcommand {\im}     {\hbox{\rm Im}\,}

\newcommand{\SU}{\mathrm{SU}}
\newcommand{\su}{\mathfrak{su}}
\newcommand{\str}{\mathrm{str}}
\newcommand{\uni}{\mathfrak{u}}
\newcommand{\tr}{\mathrm{tr}}
\newcommand{\h}{\mathfrak{h}}

\begin{document}

\title[Toeplitz operators on the super unit ball]{Commutative $C^*$-algebras generated by Toeplitz operators on the super unit ball}
\author{R.~Quiroga-Barranco}
\address{Centro de Investigaci\'on en Matem\'aticas \\
Guanajuato \\
Mexico}
\email{quiroga@cimat.mx}
\author{A.~S\'anchez-Nungaray}
\address{Facultad de Matem\'aticas \\
Universidad Veracruzana \\
Veracruz \\
Mexico}
\email{armsanchez@uv.mx}

\begin{abstract}
    We extend known results about commutative $C^*$-algebras generated Toeplitz operators over the unit ball to the supermanifold setup. This is obtained by constructing commutative $C^*$-algebras of super Toeplitz operators over the super ball $\mathbb{B}^{p|q}$ and the super Siegel domain $\mathbb{U}^{p|q}$ that naturally generalize the previous results for the unit ball and the Siegel domain. In particular, we obtain one such commutative $C^*$-algebra for each even maximal Abelian subgroup of automorphisms of the super ball.
\end{abstract}

\thanks{The authors were supported by SNI and  Conacyt grants.}

\maketitle

\section{Introduction}

In \cite{GQV} it was proved, under mild conditions, that a $C^*$-algebra generated by Toeplitz operators is commutative on each weighted Bergman space of the unit disk if and only if there is a pencil of hyperbolic geodesics of the unit disk such that the symbols of the Toeplitz operators are constant on the cycles of this pencil. In fact, the cycles are the orbits of a one-parameter subgroup of isometries for the hyperbolic geometry on the unit disk. We note that there are three different non-conjugate model classes of such subgroups: elliptic, parabolic and hyperbolic. This provides us with the following scheme: the $C^*$-algebra generated by Toeplitz operators is commutative on each weighted Bergman space on the unit disk if and only if there is a maximal Abelian subgroup of Mobius transformations such that the symbols of the Toeplitz operators are invariant under the action of this subgroup.

A generalization of this scheme was given in \cite{QV1,QV2,QV3}. The generalization is obtained by considering a maximal Abelian subgroup of biholomorphisms of the unit ball, then the $C^*$-algebra generated by the Toeplitz operators whose symbols are invariant under the action of such subgroup is commutative on each weighted Bergman space. It was also noted that there are five different non-conjugate model classes of such subgroups: quasi-elliptic, quasi-parabolic, quasi-hyperbolic, nilpotent, and quasi-nilpotent. We refer to the above mentioned works for further details.

On the other hand, the general theory of non-perturbative quantization for a class of Hermitian symmetric supermanifolds (a particular case is the super ball) was developed in  \cite{BKLR1,BKLR2}. Such quantization is based on the notion of super Toeplitz operator defined on a suitable $\mathbb{Z}_2$-graded Hilbert space of superholomorphic functions that effectively defines super Bergman spaces. These quantized supermanifolds yield the $C^*$-algebra generated by such super Teoplitz operators. Along these lines, in \cite{LU} it is given an exhaustive description of the super Toeplitz operators over the super ball using classical Toeplitz-type operators.

Recently and quite unexpectedly it was observed in \cite{L-SN1,L-SN2,SN1} that there are five different non-conjugate classes of maximal abelian supergroups of isomorphisms of the super disk labeled by the names super-elliptic, quasi-elliptic, super-parabolic, quasi-parabolic and quasi-hyperbolic. In these works it is proved that the $C^*$-algebra of super Toeplitz operators whose symbols are invariant under the action of one of these subgroups is commutative on each weighted super Bergman space.

The main goal of this work is to extend the previous results and theory to the case of the super unit ball $\mathbb{B}^{p|q}$ and its unbounded realization $\mathbb{U}^{p|q}$, the super Siegel domain.

Hence, we introduce in Section~\ref{sec:superBergman} the super Bergman space for the super Siegel domain and prove its unitary equivalence with the super Bergman space of the super ball. This allows us to use the known theory to compute the super Bergman projection and define the super Toeplitz operators on the super Siegel domain.

With the above setup, we obtain the full list of even maximal Abelian subgroups of the group of automorphisms of the super unit ball. Our classification is based on the analysis of the corresponding maximal Abelian subalgebras that are described in Theorem~\ref{thm:MASA}. This is used in Section~\ref{sec:MASA} to give an explicit description of the actions of the even maximal Abelian subgroups on the super unit ball. Most of these are easier to present in the super Siegel domain. There exists 5 non-equivalent types of even maximal Abelian subgroups for our supermanifold setup, one of them depending on a parameter for a total of $n+2$ different conjugacy classes. We label the five types with the names quasi-elliptic, quasi-parabolic, nilpotent, quasi-hyperbolic, and quasi-nilpotent.

Section~\ref{sec:superBargmann} introduces a super Bargmann transform corresponding to each one of the conjugacy classes of even maximal Abelian subgroups mentioned above. These super Bargmann transforms generalize those presented in \cite{QV2}. At the same time, our transforms allow us to prove that the $C^*$-algebra generated by the Toeplitz operators whose symbols are invariant by one of the even maximal Abelian subgroups is commutative. This is obtained in Section~\ref{sec:commToeplitz} and is thus the core of this work. The relevant results, for each type of even maximal Abelian subgroups, are Theorems~\ref{thm:q-elliptic}, \ref{thm:q-parabolic}, \ref{thm:nilpotent}, \ref{thm:q-nilpotent}, \ref{thm:q-hyperbolic}. We note that our results have the same strength of those for the classical case in that our approach using a super Bargmann transform allows us to realize the commuting Toeplitz operators in each case as multiplication operators.

\section{Weighted super Bergman spaces and projections}
\label{sec:superBergman}

For $p \geq 1$, let $\mathcal{O}(\mathbb{B}^p )$ denote the algebra of holomorphic functions $\psi(z_1,\ldots,z_p)$ on the open unit ball
$$
    \mathbb{B}^p= \left\{z=(z_1,\ldots,z_p): |z|^2=|z_1|^2+\cdots+|z_p|^2<1    \right\}
$$
in $\mathbb{C}^p$.

For $p \geq 1$, let $\mathcal{O}(\mathbb{U}^p )$ denote the algebra of all functions $\psi(w_1,\ldots,w_p)$ that are holomorphic on the Siegel domain
               $$\mathbb{U}^p= \left\{w=(w',w_p)=(w_1,\ldots,w_p)\in \mathbb{C}^p :     \text{Im}(w_p)-|w'|^2>0  \right\},$$
where $w'\in \mathbb{C}^{p-1}$.

\begin{defn}
For $\nu> p$, the weighted Bergman space
               $$H^2_{\nu}(\mathbb{B}^p )=\mathcal{O}(\mathbb{B}^p )\cap L^2(\mathbb{B}^p,d\mu_{\nu} )$$
consists of all holomorphic functions on $\mathbb{B}^p $ which are square-integrable for the probability measure
$$d\mu_{\nu}=c_{\nu}(1-z\bar{z})^{\nu-p-1} dz, \ \ \nu>p,$$
where the normalizing constant is given by
$$c_{\nu}=\frac{\Gamma(\nu)}{\pi^p\Gamma(\nu-p)}$$
with $z \bar{w} =z_1\bar{w}_1+\cdots+z_p\bar{w}_p$ on $\mathbb{C}^p$ and $dz$ is the Lebesgue measure.
Correspondingly, the weighted Bergman space
               $$H^2_{\nu}(\mathbb{U}^p )=\mathcal{O}(\mathbb{U}^p )\cap L^2(\mathbb{U}^p,d\tilde{\mu}_{\nu} )$$
consists of all holomorphic functions on $\mathbb{U}^p $ which are square-integrable for the probability measure
$$d\tilde{\mu}_{\nu}=\frac{c_{\nu}}{4}(\text{Im}(w_p)-w' \bar{w}')^{\nu-p-1} dz,\ \ \nu>p,$$
where $dz$ is the Lebesgue measure.
\end{defn}

It is well-known (see \cite{Z}) that $H^2_{\nu}(\mathbb{B}^p )$ has the reproducing kernel
$$ K_{\mathbb{B}^p,\nu}(z,w)=(1- z \bar{w})^{-\nu},$$
for all $z, w \in \mathbb{B}^p$.
And that  $H^2_{\nu}(\mathbb{U}^p )$ has the reproducing kernel
$$ K_{\mathbb{U}^p,\nu}(z,w)=\left( \frac{z_n-\bar{w}_n}{2i}-z'\bar{w'}  \right)^{-\nu},  $$
for all $z, w \in \mathbb{U}^p$.

Let $\Lambda_q$ denote the complex Grassmann algebra with the generators $\xi_1,\ldots,\xi_q$
satisfying the relations
   $$\xi_i\xi_j+\xi_j\xi_i=0,$$
for $1\leq i, j \leq q$. If we take $Q := \{1, \ldots , q\}$, then we have
              $$\Lambda_q= \mathbb{C}\langle  \xi_I: I\subset Q \rangle,$$
where $\xi_I=\xi_{i_1}\cdots\xi_{i_k}$
if $I = \{i_1 <\cdots < i_k \}$. For disjoint subsets $I, J$ we have
                   $\xi_I\xi_J=\varepsilon_{I,J}\xi_{I\cup J}$
for some $\varepsilon_{I,J} = \pm 1$ whose value depends on the pair $I,J$.

The tensor product algebra
              $$\mathcal{O}(\mathbb{B}^{p|q} ):=\mathcal{O}(\mathbb{B}^p )\otimes \Lambda_q$$
consists of all super-holomorphic functions
               $$\Psi(z,\xi)=\sum_{I\subset Q} \Psi_I \xi_I,$$
where $\Psi_I\in \mathcal{O}(\mathbb{B}^p )$ for all $I \subset Q$.
In a similar way, the tensor product algebra
              $$\mathcal{O}(\mathbb{U}^{p|q} ):=\mathcal{O}(\mathbb{U}^p )\otimes \Lambda_q$$
consists of all super-holomorphic functions. In this case we have a similar expression for the super-holomorphic functions.

Let $\Lambda_q^{\mathbb{C}}$ denote the complex Grassmann algebra with the generators $\xi_1,\ldots,\xi_q$, $\bar{\xi}_1,\ldots,\bar{\xi}_q$
satisfying the relations
\begin{align*}
    \xi_i\xi_j+\xi_j\xi_i&=0, \\
    \bar{\xi}_i\bar{\xi}_j+\bar{\xi}_j\bar{\xi}_i&=0, \\
    \xi_i\bar{\xi}_j+\xi_j\bar{\xi}_i&=0,
\end{align*}
for $1\leq i, j \leq q$.

Thus we have
              $$\Lambda_q^{\mathbb{C}}= \mathbb{C}\langle  \xi_I \xi^*_J: I,J\subset Q \rangle,$$
where $\xi_I=\xi_{i_1}\cdots\xi_{i_k}$ and $\xi_I^*=\bar{\xi}_{j_l}\cdots\bar{\xi}_{j_1}$
if $I = \{i_1 <\cdots < i_k \}$ and $J=\{j_1\ldots,j_l\}$.

The tensor product algebra
              $$\mathcal{C}(\mathbb{B}^{p|q} ):=\mathcal{C}(\mathbb{B}^p )\otimes \Lambda_q^{\mathbb{C}}$$
consists of all continuous super functions
               $$\Psi=\sum_{I,J\subset Q} \Psi_{I,J} \xi_I\xi^*_J,$$
where $\Psi_{I,J}\in \mathcal{C}(\mathbb{B}^p )$ for all $I,J \subset Q$.
Similarly
$$\mathcal{C}(\mathbb{U}^{p|q} ):=\mathcal{C}(\mathbb{U}^p )\otimes \Lambda_q^{\mathbb{C}}$$
consists of all continuous super functions
               $$\Psi=\sum_{I,J\subset Q} \Psi_{I,J} \xi_I\xi^*_J,$$
where $\Psi_{I,J}\in \mathcal{C}(\mathbb{U}^p )$ for all $I,J \subset Q$.

There is a natural involution $\Psi \mapsto \Psi^*$ on these spaces of super functions, which is defined by
$$\Psi^*=\sum_{I,J\subset Q} \bar{\Psi}_{I,J} \xi_J\xi^*_I,$$
for $\Psi$ as above.

The Berezin integral on $\mathbb{V}^{p|q}$ is defined by
$$\int_{\mathbb{V}^{p|q}} dz d\xi F(z,\xi)=\int_{\mathbb{V}^{p}}  f_{Q,Q}(z)dz$$
for $F\in\mathcal{C}(\mathbb{V}^{p|q} )$, where the normalization is given by
 $$\int_{\mathbb{V}^{p|q}} dz d\xi \xi_Q^* \xi_Q=1,$$
where $\mathbb{V}^p$ is either $\mathbb{B}^p$ or $\mathbb{U}^p$.


For any given  morphism $\gamma$ between super domains we define
\begin{equation}\label{BerLSG}
    \gamma'(Z)=
    \text{Ber}
        \begin{pmatrix*}
            \frac{\partial w}{\partial z} & \frac{\partial \omega}{\partial z} \\
            \frac{\partial w}{\partial \xi} & \frac{\partial \omega}{\partial \xi} \\
        \end{pmatrix*}
    = \text{Ber}\frac{\partial W}{\partial Z},
\end{equation}
where $Z=(z_1,\ldots,z_p,\xi_1,\ldots, \xi_q)$, $W= (w_1,\ldots,w_p,\omega_1,\ldots, \omega_q)$ and where the Berezinian is defined as follows
\begin{equation}\label{Berezinian}
\text{Ber}
\begin{pmatrix*}
  A & B \\
  C & D \\
\end{pmatrix*}
=\det(A-BD^{-1}C) \det(D)^{-1}.
\end{equation}
We refer to \cite{B} for more details.

We now recall a natural biholomorphism of supermanifolds between the super $p$-ball $\mathbb{B}^{p|q}$ and the super Siegel domain $\mathbb{U}^{p|q}$.

We define the super Cayley transform from $\mathbb{B}^{p|q }$ to $\mathbb{U}^{p|q }$ in local coordinates by
\begin{equation}\label{Cayley}
\psi(z_1,\ldots,z_p,\xi_1,\ldots, \xi_q)= (w_1,\ldots,w_p,\omega_1,\ldots, \omega_q),
\end{equation}
where
\begin{align*}
w_k&=\frac{i z_k}{1+z_p}, \text{ for }k=1,\ldots, p-1,\\
w_p&=i\frac{1-z_p}{1+z_p},\\
\omega_k&=\frac{i \xi_k}{1+z_p}, \text{ for }k=1,\ldots, q.
\end{align*}
The inverse transform is given by
\begin{align*}
z_k&=\frac{-2i w_k}{1-iw_p}, \text{ for }k=1,\ldots, p-1,\\
z_p&=i\frac{1+iw_p}{1-iw_p},\\
\xi_k&=\frac{-2i \omega_k}{1-iw_p}, \text{ for }k=1,\ldots, q.
\end{align*}

\begin{lem}\label{lema-d-p-form}
Let $Z=(z,\xi)$, $U=(u,\eta)$,  $\psi(Z)=(w,\omega)=W$
and $\psi(U)=(v,\zeta)=V$ where $\psi$ is given by (\ref{Cayley}). Then, we have
\begin{equation}\label{exp-inv-1-p-d}
    (1+z\bar{u}-\xi\bar{\eta})(1+z_p)^{-1}  \overline{ (1+u_p)^{-1} } =\left(\frac{ w_p-\bar{v_p}}{2i}-w'\bar{v'}- \omega\bar{\zeta} \right),
\end{equation}
\begin{equation}\label{exp-inv-2p-d}
   (1+z\bar{u}-\xi\bar{\eta}) =\left(\frac{ w_p-\bar{v_p}}{2i}-w'\bar{v'}- \omega\bar{\zeta}\right)   4 (1-iw_p)^{-1}  \overline{ (1-iv_p)^{-1} },
\end{equation}
where $\xi\bar{\theta}=\xi_1\bar{\theta_1}+\cdots+\xi_q\bar{\theta}_q$   and $z\bar{w}=z_1\bar{w_1}+\cdots+z_p\bar{w}_p$.
\end{lem}
\begin{proof}
Using the above and substituting $\psi(Z)=(w,\omega)=W$
and $\psi(U)=(v,\zeta)=V$ on the left hand side of equation (\ref{exp-inv-1-p-d}) it follows that
\begin{align*}
&\left(\frac{ w_p-\bar{v_p}}{2i} - w'v'-\omega\bar{\zeta}\right)\\
=&
\left(\frac{1}{2i}\left[i\frac{1-z_p}{1+z_p}-\overline{i\frac{1-u_p}{1+u_p}}\right] -
\sum_{k=1}^{p-1}\frac{i z_k}{1+z_p}\overline{\frac{i u_k}{1+u_p}}-
\sum_{k=1}^{q}\frac{i \xi_k}{1+z_p}\overline{\frac{i \eta_k}{1+u_p}}\right)\\
=&  \left(\frac{1}{2}\left[(1-z_p)(1+\overline{u_p})  +(1-\overline{u_p})(1+z_p)\right] -
\sum_{k=1}^{p-1}  z_k \overline{u_k} -
\sum_{k=1}^{q} \xi_k \overline{ \eta_k} \right)\\
&\times \frac{1}{(1+z_p)}\frac{1}{(1+\bar{u_p})}\\
=&    \frac{1}{(1+z_p)}\frac{1}{(1+\bar{u_p})} \left(1-z_p\overline{u_p}- \sum_{k=1}^{p-1}  z_k \overline{v_k} -
\sum_{k=1}^{q} \xi_k \overline{ \eta_k} \right)\\
=&(1+z_p)^{-1}  \overline{ (1+u_p)^{-1} } (1+z\bar{u}-\xi\bar{\eta}).
\end{align*}
Its clear that $1+z_p=2(1-iw_p)^{-1}$, then we obtain \eqref{exp-inv-2p-d}.
\end{proof}


\begin{defn}
 For any parameter $\nu > p - q + 1$, the (weighted) super-Bergman
space
$$H^2_{\nu}(\mathbb{B}^{p|q })\subset\mathcal{O}(\mathbb{B}^{p|q} )$$
consists of all super-holomorphic functions $\Psi(z,\xi)$ which satisfy the square-integrability condition
\begin{align*}
&(\Psi|\Psi)_{\mathbb{B}^{p|q},\nu} = \\
&\frac{\Gamma(\nu)}{\pi^p\Gamma(\nu+q-p)} \int_{\mathbb{B}^{p|q}} dz d\xi (1-z\bar{z}-\xi\bar{\xi})^{\nu+q-p-1} \Psi(z,\xi)^*\Psi(z,\xi)<\infty,
\end{align*}
where $\xi\bar{\theta}=\xi_1\bar{\theta_1}+\cdots+\xi_q\bar{\theta}_q$   and $z\bar{w}=z_1\bar{w_1}+\cdots+z_p\bar{w}_p$.
Moreover
$$\frac{(1-z\bar{z}-\xi\bar{\xi})^{\nu+q-p-1}}{\pi^p\Gamma(\nu+q-p)} =\sum_{J\subset Q} \frac{(1-z\bar{z})^{\nu+|J|-p-1}}{\pi^p\Gamma(\nu+|J|-p)} \xi_{Q\setminus J}^*\xi_{Q\setminus J}$$
\end{defn}

In \cite{LU} the authors proved that the super Bergman space has a decomposition in direct sum of classical Bergman space and gave an explicit expression for the super Bergman projection. If we take  $\Psi=\sum_{M\subset Q} \psi_M \xi_M \in H^2_{\nu}(\mathbb{B}^{p|q })
\subset\mathcal{O}(\mathbb{B}^{p|q} )$ then the inner product has the form
\begin{align} \nonumber
& \frac{\Gamma(\nu)}{\pi^p\Gamma(\nu+q-p)} \int_{\mathbb{B}^{p|q}} dz d\xi (1-z\bar{z}-\xi\bar{\xi})^{\nu+p-p-1} \Psi(z,\xi)^*\Psi(z,\xi) = \\
\label{bergman-descoposition}
&\sum_{m=1}^q    \frac{\Gamma(\nu)}{\Gamma(\nu+m)}   \sum_{M\subset Q, |M|=m}  \| \psi_M(z) \|^2_{\mathbb{ B}^{p},\nu+m}.
\end{align}
Thus the super Bergman space over the super ball  has an orthogonal decomposition
                          $$H^2_{\nu}(\mathbb{B}^{p|q })= \sum_{m=0}^q H^2_{\nu+m}(\mathbb{B}^{p })\otimes\Lambda^m(\mathbb{C^q}) $$
into a sum of weighted Bergman spaces for $0 \leq m \leq q$, with multiplicity $\binom{m}{q}$. Moreover, this super Bergman space has  the reproducing kernel property
\begin{align*}
    P_{\mathbb{B}^{p|q },\nu}\Psi(z,\xi)=
   &\frac{\Gamma(\nu)}{\pi^p\Gamma(\nu+q-p)} \int_{\mathbb{B}^{p|q}} dw d\omega (1-w\bar{w}-\omega\bar{\omega})^{\nu+q-p-1}\\
   &\times (1-z\bar{w}-\xi\bar{\omega})^{-\nu} \Psi(w,\omega)=\Psi(w,\omega).
\end{align*}
In other words, $H^2_{\nu}(\mathbb{B}^{p|q })$ has the reproducing kernel
$$K_{\mathbb{B}^{p|q},\nu} (z, \xi, w, \omega) =(1-z\bar{w}-\xi\bar{\omega})^{-\nu}.$$

\begin{defn}
For any parameter $\nu > p - q + 1$, the (weighted) super-Bergman
space
$$H^2_{\nu}(\mathbb{U}^{p|q })\subset\mathcal{O}(\mathbb{U}^{p|q} )$$
consists of all super-holomorphic functions $\Psi(w,\omega)$ which satisfy the square-integrability condition
\begin{align*}
(\Psi|\Psi)_{\mathbb{U}^{p|q},\nu}=&
\frac{\Gamma(\nu)}{4 \pi^p\Gamma(\nu+q-p)}
\int_{\mathbb{U}^{p|q}} dw d\omega ( \text{Im}(w_p)-w'\bar{w'}- \omega\bar{\omega})^{\nu+q-p-1} \\ & \times\Psi(z,\omega)^*\Psi(w,\omega)<\infty,
\end{align*}
where $\xi\bar{\omega}=\xi_1\bar{\omega}_1+\cdots+\xi_q\bar{\omega}_q$   and $z'\bar{w}'=z_1\bar{w}_1+\cdots+z_{p-1}\bar{w}_{p-1}$. Where one can prove that
\begin{align*}
&  \frac{1}{4\pi^p\Gamma(\nu+q-p)}\left(\frac{ w_p-\bar{w_p}}{2i}-w'\bar{w'}- \omega\bar{\omega}\right)^{\nu+q-p-1}\\
&=
\sum_{J\subset Q} \frac{1}{4\pi^p\Gamma(\nu+|J|-p)}\left(\frac{ w_p-\bar{w_p}}{2i}-w'\bar{w'}\right)^{\nu+|J|-p-1}
\omega_{Q\setminus J}^*\omega_{Q\setminus J}.
\end{align*}
\end{defn}

We now observe that the Berezinian of the Jacobian matrix of the transformation $\psi$ given by \ref{Cayley} satisfies
$$\psi'(Z)=\text{Ber}
    \begin{pmatrix*}
              A &  B \\
              0 & C \\
    \end{pmatrix*}
=\det(A) \cdot\det(C^{-1}),$$
where $$A=\text{Diagonal}( i(1+z_p)^{-1},\ldots, i(1+z_p)^{-1},-2i (1+z_p)^{-2})$$ and $$C=\text{Diagonal}( i(1+z_p)^{-1},\ldots, i(1+z_p)^{-1}).$$
As a consequence we obtain that $\psi'(Z)= -2 i^{p-q+1}(1+z_p)^{q-p-1}$ and analogously $(\psi^{-1})'(W)= -2^{p-q} i^{q-p-1}(1-iw_p)^{q-p-1}$.

\begin{defn}
We define the operator
$U_{\nu}:H^2_{\nu}(\mathbb{B}^{p|q })\rightarrow H^2_{\nu}(\mathbb{U}^{p|q })$
given by
$$U_{\nu}(\Psi)(W)=\Psi(\psi^{-1}(W)) \left( \frac{2}{1-iw_p}  \right)^{\nu},$$
and its adjoint
$U^*_{\nu}: H^2_{\nu}(\mathbb{U}^{p|q })\rightarrow H^2_{\nu}(\mathbb{B}^{p|q })$
given by
$$U^*_{\nu}(\Psi)(Z)=\Psi(\psi(Z)) \left( \frac{1}{1+z_p}  \right)^{\nu}.$$
\end{defn}

\begin{thm}\label{unitary-B-S}
 The operador $U_{\nu}:H^2_{\nu}(\mathbb{B}^{p|q })\rightarrow H^2_{\nu}(\mathbb{U}^{p|q })$ is unitary.
\end{thm}
\begin{proof}
 Consider $\Psi\in H^2_{\nu}(\mathbb{B}^{p|q })$. Then we have
\[
(\Psi|\Psi)_{\mathbb{B}^{p|q},\nu}
=\frac{\Gamma(\nu)}{\pi^p\Gamma(\nu+q-p)} \int_{\mathbb{B}^{p|q}} dz d\xi (1-z\bar{z}-\xi\bar{\xi})^{\nu+q-p-1} \Psi(z,\xi)^*\Psi(z,\xi).
\]

Using Lemma \ref{lema-d-p-form} and the change of variable given by the super Cayley transform \eqref{Cayley} on the right side of the above equation, we obtain
\begin{align*}
&\frac{\Gamma(\nu)}{\pi^p\Gamma(\nu+q-p)}
\int_{\mathbb{U}^{p|q}} dw d\omega   (\psi^{-1})'(W) \overline{(\psi^{-1})'(W)}
  \\
&\times  (\text{Im}(w_p)-w'\bar{w'}- \omega\bar{\omega})^{\nu+q-p-1}4^{\nu+q-p-1} ((1-iw_p)^{-1} )^{\nu+q-p-1}   \\
&\times  \overline{ (1-iw_p)^{-1} })^{\nu+q-p-1}    \Psi(\psi^{-1}(W))^*\Psi(\psi^{-1}(W)) \\
=&\frac{\Gamma(\nu)}{\pi^p\Gamma(\nu+q-p)}
\int_{\mathbb{U}^{p|q}} dw d\omega   4^{p-q} (1-iw_p)^{q-p-1} \overline{ (1-iw_p)^{q-p-1}}  \\
&\times \left(\frac{ w_p-\bar{w_p}}{2i}-w'\bar{z'}- \omega\bar{\omega}\right)^{\nu+q-p-1}   4^{\nu+q-p-1} ((1-iw_p)^{-1} )^{\nu+q-p-1}  \\
&\times     \overline{ (1-iw_p)^{-1} })^{\nu+q-p-1}    \Psi(\psi^{-1}(W))^*\Psi(\psi^{-1}(W))  \\
=&\frac{\Gamma(\nu)}{4\pi^p\Gamma(\nu+q-p)}
\int_{\mathbb{U}^{p|q}} dw d\omega    \left(\frac{ w_p-\bar{w_p}}{2i}-w'\bar{z'}- \omega\bar{\omega}\right)^{\nu+q-p-1} \\
&\times  [\Psi(\psi^{-1}(W)) \left( \frac{2}{1-iw_p}  \right)^{\nu} ]^*[\Psi(\psi^{-1}(W))\left( \frac{2}{1-iw_p}  \right)^{\nu}] \\
=&(U_{\nu}(\Psi)|U_{\nu}(\Psi))_{\mathbb{U}^{p|q},\nu}.
\end{align*}
\end{proof}


\begin{cor}
If $\Psi=\sum_{M\subset Q} \psi_M \xi_M \in \mathcal{O}_{\nu}(\mathbb{U}^{p|q })$, then we have
\begin{align*}
\frac{\Gamma(\nu)}{\pi^p\Gamma(\nu+q-p)}
\int_{\mathbb{U}^{p|q}} dw d\omega &\left(\frac{ w_p-\bar{w_p}}{2i}-w'\bar{w'}- \omega\bar{\omega}\right)^{\nu+q-p-1} \\ &\times \Psi(w,\omega)^*\Psi(w,\omega) \\
=&\sum_{m=1}^q    \frac{\Gamma(\nu)}{\Gamma(\nu+m)}   \sum_{M\subset Q, |M|=m}  \| \psi_M(z) \|^2_{\mathbb{U}^{p},\nu+m}.
\end{align*}
In other words, there is an orthogonal decomposition
$$H^2_{\nu}(\mathbb{U}^{p|q })= \sum_{m=0}^q H^2_{\nu+m}(\mathbb{U}^{p })\otimes\Lambda^m(\mathbb{C}^q) $$
into a sum of weighted Bergman spaces for $0 \leq m \leq q$, whose corresponding multiplicities are $\binom{m}{q}$.
\end{cor}
\begin{proof}
We rewrite the operator $U_{\nu}$ as follows
\begin{align*}
U_{\nu}(\Psi)(W)&= \sum_{M\subset Q}(-i)^{|M|} \psi_M( \psi_0^{-1}(w) ) \left( \frac{2}{1-iw_p}  \right)^{\nu+|M|}  \omega_M\\
& = \sum_{M\subset Q}(-i)^{|M|} V_{\nu+|M|}(\psi_M)(w) \omega_M,
\end{align*}
where $V_{\nu+m}:H^2_{\nu+m}(\mathbb{B}^{p })\rightarrow H^2_{\nu+m}(\mathbb{U}^{p })$ is defined by
$$V_{\nu+m}(f)=f( \psi_0^{-1}(w) ) \left( \frac{2}{1-iw_p}  \right)^{\nu+|M|}.$$
Since we known that $V_{\nu+m}$ is a unitary operator, the result follows.
\end{proof}


\begin{prop}
For $\nu > p$ and $\Psi=\sum_{M\subset Q} \psi_M \xi_M \in \mathcal{O}_{\nu}(\mathbb{U}^{p|q })$, we have the reproducing kernel property
\begin{align*}
 P_{\mathbb{U}^{p|q },\nu}\Psi(w,\omega)=&\frac{\Gamma(\nu)}{\pi^p\Gamma(\nu+q-p)} \int_{\mathbb{U}^{p|q}} dv d\zeta
\left(\frac{ v_p-\bar{v_p}}{2i}-v'\bar{v'}- \zeta\bar{\zeta}\right)^{\nu+q-p-1} \\
&\times\left(\frac{ w_p-\bar{v_p}}{2i}-w'\bar{v'}- \omega\bar{\zeta} \right)^{-\nu} \Psi(v,\zeta)=\Psi(w,\omega).
\end{align*}
In particular, $H^2_{\nu}(\mathbb{U}^{p|q })$ has the reproducing kernel
$$K_{\mathbb{U}^{p|q},\nu} (w, \omega, v, \zeta) =\left(\frac{ w_p-\bar{v_p}}{2i}-w'\bar{v'}- \omega\bar{\zeta} \right)^{-\nu}.$$
\end{prop}

\begin{proof}

We know that $U_{\nu}$ is unitary, therefore the Bergman projection over $H^2_{\nu}(\mathbb{U}^{p|q })$ is given
by $$P_{\nu,\mathbb{U}^{p|q }}=U_{\nu}P_{\nu,\mathbb{B}^{p|q }}U^*_{\nu},$$
thus $P_{\mathbb{U}^{p|q },\nu} (\Psi)$ is given by
\begin{equation*}
  U_{\nu}\left[ \frac{\Gamma(\nu)}{\pi^p\Gamma(\nu+q-p)} \int_{\mathbb{B}^{p|q}} du d\eta (1-u\bar{u}-\eta\bar{\eta})^{\nu+q-p-1} (1-z\bar{u}-\xi\bar{\eta})^{-\nu} U^*_{\nu}(\Psi(V))\right].
\end{equation*}

Using the change of variable \eqref{Cayley}, taking
$W=(w,\omega)=\psi(z,\xi)=\psi(Z)$, $V=(v,\zeta)=\psi(u,\eta)=\psi(U)$ and
Lemma~\ref{lema-d-p-form}, we obtain
\begin{align*}
& P_{\mathbb{U}^{p|q },\nu} (\Psi)\\
=& \left( \frac{2}{1-iw_p}  \right)^{\nu} \frac{\Gamma(\nu)}{4\pi^p\Gamma(\nu+q-p)} \int_{\mathbb{U}^{p|q}} dv d\zeta
\left(\frac{ v_p-\bar{v_p}}{2i}-v'\bar{v'}- \zeta\bar{\zeta}\right)^{\nu+q-p-1}
 \\
&\times   (2(1-iv_p)^{-1} )^{\nu}   \overline{ (2(1-iv_p)^{-1} })^{\nu} \Psi(V) \left( \frac{1}{1+u_p}  \right)^{\nu} \\
&\times\left[  \left(\frac{ w_p-\bar{v_p}}{2i}-w'\bar{v'}- \omega\bar{\zeta}\right)   4 (1-iw_p)^{-1}  \overline{ (1-iv_p)^{-1} }\right]^{-\nu}\\
=&  \frac{\Gamma(\nu)}{4\pi^p\Gamma(\nu+q-p)} \int_{\mathbb{U}^{p|q}} dv d\zeta   (\frac{ v_p-\bar{v_p}}{2i}-v'\bar{v'}- \zeta\bar{\zeta})^{\nu+q-p-1} \\
&\times   \left(\frac{ w_p-\bar{v_p}}{2i}-w'\bar{v'}- \omega\bar{\zeta}    \right)^{-\nu} \Psi(V)
   \left( \frac{1}{1+u_p}  \right)^{\nu}  \left( \frac{2}{1-iv_p}  \right)^{\nu}\\
=&\frac{\Gamma(\nu)}{4\pi^p\Gamma(\nu+q-p)} \int_{\mathbb{U}^{p|q}} dv d\zeta   (\frac{ v_p-\bar{v_p}}{2i}-v'\bar{v'}- \zeta\bar{\zeta})^{\nu+q-p-1}\\
&\times
 \left(\frac{ w_p-\bar{v_p}}{2i}-w'\bar{v'}- \omega\bar{\zeta}    \right)^{-\nu} \Psi(V).
\end{align*}
\end{proof}

\section{The super-group $\SU(p,1|q)$ and its even MASG's}
\label{sec:MASA}
There is a super Lie group denoted by $\SU(p,1|q)$ that is naturally associated to the super unit ball $\mathrm{B}^{p,q}$ and whose definition we now recall (see \cite{BKLR1}). The base manifold of $\SU(p,1|q)$ is the Lie group $\SU(p,1) \times \SU(q)$. For the structure sheaf we use the Grassmann algebra $\Lambda(M_{p+q+1}(\C))$, where $M_{p+q+1}(\C)$ is the space of complex square matrices of size $(p+q+1)\times(p+q+1)$, and we consider the tensor product
\[
    C^\infty(\SU(p,1)) \otimes \Lambda(M_{p+q+1}(\C)).
\]
The variables corresponding to the matrix entries are given the following parity assignments
$$
    p(\gamma_{jk}) = p(\overline{\gamma}_{jk}) =
    \left\{
    \begin{array}{ll}
        0, & \text{if }1 \leq j, k \leq p+ 1 \text{ or }p+1 < j, k \leq p+q+ 1, \\
        1, & \text{otherwise}.
    \end{array}
        \right.
$$
Thus we have that the super-matrix has a  natural block  decomposition as follows
$$
\gamma =
    \begin{pmatrix}
        A & C \\
        D & B
    \end{pmatrix},
$$
where $A$ and $B$ are even square matrices with sizes $(p+1)\times(p+1)$ and $q\times q$, respectively, and $C$ and $D$ are odd matrices with sizes $q \times (p+1)$ and $(p+1)\times q$, respectively. The structure sheaf of $\SU(p,1|q)$ is obtained by considering the set of matrices $\gamma$ as above that satisfy
$$ \text{Ber}\gamma = 1, \quad \gamma^* J_{p,1|q} \gamma = J_{p,1|q}, $$
where
$$
J_{p,1|q}=\begin{pmatrix}
    I_p &0 &0 \\
     0&-1 & 0\\
     0&0 & -I_q
    \end{pmatrix}.
$$
In particular, the Lie super algebra $\su(p,1|q)$ of $\SU(p,1|q)$ is given by the set of matrices $\gamma$ that satisfy the conditions
\[
    \str(\gamma) = 0, \quad \gamma^* J_{p,1|q} + J_{p,1|q} \gamma^* = 0,
\]
where $\str(\gamma)=\tr(A)-\tr(B)$ and $\tr(A)$ are the usual supertrace and trace, respectively.
Since the parity of the entries for such matrices is defined as above, we conclude that the Lie algebra of even elements of $\su(p,1|q)$ is given by
\[
    \su(p,1|q)_0 =
    \left\{
        \begin{pmatrix*}
            A & 0 \\
            0 & B
        \end{pmatrix*} : A \in \uni(p,1), B \in \uni(q), \tr(A) = \tr(B)
    \right\}.
\]

On the other hand, the super Siegel domain realization $\mathbb{U}^{p|q}$ of the super unit ball $\mathbb{B}^{p|q}$ together with the super Cayley transform introduced before yield another realization of the super Lie group $\SU(p,1|q)$. More precisely, we obtain the super Lie group $\SU(K_{p|q})$ whose base Lie group is
\[
    \SU(K_p) \times \SU(q),
\]
where $\SU(K_p)$ is the Lie group of unitary transformation of the pseudo-Hermitian form on $\C^{p+1}$ whose matrix is the following (see \cite{QV3} for comparison)
\[
    K_p =
    \begin{pmatrix*}
        2I_{p-1} & 0 & 0 \\
        0 & 0 & -i \\
        0 & i & 0
    \end{pmatrix*}.
\]
The structure sheaf of $\SU(K_{p|q})$ is given by the set of matrices $\gamma$ as above that now satisfy
\[
    \text{Ber}(\gamma) = 1, \quad \gamma^* K_{p|q} \gamma = K_{p|q},
\]
where
\[
    K_{p|q} =
    \begin{pmatrix*}
        K_p & 0 \\
        0 & -I_q
    \end{pmatrix*}.
\]
In particular, the super Lie algebra of $\SU(K_{p|q})$, denoted by $\su(K_{p|q})$, has the following Lie algebra as space of even elements
\[
    \su(K_p|q)_0 =
    \left\{
        \begin{pmatrix*}
            A & 0 \\
            0 & B
        \end{pmatrix*} : A \in \uni(K_p), B \in \uni(q), \tr(A) = \tr(B)
    \right\}.
\]
Here, we have used the notation where $\uni(K_p)$ denotes the Lie algebra of the Lie group $\mathrm{U}(K_p)$ of unitary transformations for the pseudo-Hermitian product with matrix $K_p$.

We note that $\SU(p,1|q)$, $\SU(K_{p|q})$ and their Lie algebras are conjugated through the super Cayley transform considered before. At the base manifold level this implies that $\su(p,1)$ and $\su(K_p)$ are conjugate.

\begin{defn}
    An even maximal Abelian subalgebra of $\su(p,1|q)$, or an even MASA for short, is a maximal Abelian subalgebra of the Lie algebra $\su(p,1|q)_0$. An even maximal Abelian subgroup of $\SU(p,1|q)$, or an even MASG, is a connected super Lie subgroup of $\SU(p,1|q)$ whose Lie super algebra is an even MASA.
\end{defn}

We now list the collection of all even MASA of $\su(p,1|q)$ up to conjugacy. Note that since $\su(p,1|q)$ and $\su(K_{p|q})$ are conjugated the even MASA of both super Lie algebras correspond to each other. In particular, the conjugacy classes of even MASA's can be described in terms of either one of these super Lie algebras. Also note that it follows from the above remarks that the maximal Abelian subalgebras of $\su(p,1)$ and $\su(K_p)$ correspond to each other as well.

\begin{thm}
    \label{thm:MASA}
    For every even MASA $\h$ of $\su(p,1|q)$ there exist maximal Abelian subalgebras $\h_1 \subset \uni(p,1)$ and $\h_2 \subset \uni(q)$ such that $\h$ is conjugate to the even MASA
    \[
        \h_1 \boxtimes_0 \h_2 =
        \left\{
            \begin{pmatrix*}
                A & 0 \\
                0 & B
            \end{pmatrix*} : A \in \h_1, B \in \h_2, \tr(A) = \tr(B)
        \right\}.
    \]
    Furthermore, we can assume that $\h_2$ is the Lie subalgebra of diagonal matrices in $\uni(q)$ and that $\h_1$ is given by one of the following where $D(k)$ denotes the Lie algebra of $k\times k$ diagonal matrices with pure imaginary entries.
    \begin{enumerate}
        \item Quasi-elliptic: The Lie subalgebra of diagonal matrices in $\uni(p,1)$.
        \item Quasi-parabolic: The Lie subalgebra of $\su(K_p)$ that consists of the matrices of the form
            \[
                \begin{pmatrix*}
                    D & 0 & 0 \\
                    0 & z & 0 \\
                    0 & 0 & \overline{z}
                \end{pmatrix*},
            \]
            where $D \in D(p-1)$, $z \in \C$ and $\tr(D) + 2i \im(z) = 0$.
        \item Quasi-hyperbolic: The Lie subalgebra of $\su(K_p)$ that consists of the matrices of the form
            \[
                \begin{pmatrix*}
                    D & 0 & 0 \\
                    0 & iy & a \\
                    0 & 0 & iy
                \end{pmatrix*},
            \]
            where $D \in D(p-1)$, $a,y \in \R$ and $\tr(D) + 2iy = 0$.
        \item Nilpotent: The Lie subalgebra of $\su(K_p)$ that consists of the matrices of the form
            \[
                \begin{pmatrix*}
                    0 & 0 & b^t \\
                    2ib & 0 & a \\
                    0 & 0 & 0
                \end{pmatrix*},
            \]
            where $a \in \R$ and $b \in \R^{p-1}$.
        \item Quasi-nilpotent: For some $k$ such that $1 \leq k \leq p-2$, the Lie subalgebra of $\su(K_p)$ that consists of the matrices of the form
            \[
                \begin{pmatrix*}
                    D & 0 & 0 & 0 \\
                    0 & iy I_{p-k-1} & 0 & b^t \\
                    0 & 2ib & iy & a \\
                    0 & 0 & 0 & iy
                \end{pmatrix*},
            \]
            where $a,y \in \R$, $b \in \R^{p-k-1}$, $D \in D(k)$ and $\tr(D) + iy(p-k+1) = 0$.
    \end{enumerate}
\end{thm}
\begin{proof}
    Let $\h$ be an even MASA subalgebra of $\su(p,1|q)$. In particular, $\h$ is a MASA (maximal Abelian subalgebra) of $\su(p,1|q)_0$. We first note that
    \[
        \su(p,1|q)_0 \simeq \su(p,1) \times \uni(q)
    \]
    as Lie algebras where the isomorphism is given by the assignment
    \[
        \begin{pmatrix*}
            A & 0 \\
            0 & B
        \end{pmatrix*}
            \mapsto
                \left(A - \frac{1}{p+1} \tr(A) I_{p+1}, B \right).
    \]
    And so, we can consider $\h$ as a MASA of $\su(p,1) \times \uni(q)$.

    Let $\h_1$ and $\h_2$ be the projections of $\h$ into the first and second factors, respectively. Hence, $\h_1$ and $\h_2$ are both Abelian Lie algebras. It is clear that $\h \subset \h_1 \times \h_2$ and the maximality of $\h$ implies that $\h = \h_1 \times \h_2$. Furthermore, the same argument shows that $\h_1$ and $\h_2$ are MASA's of $\su(p,1)$ and $\uni(q)$, respectively. It is well known that there is a single conjugacy class of MASA's of $\uni(q)$ with a representative given by $D(q)$. The conjugacy classes of MASA's of $\su(p,1)$ are also known and they are listed in \cite{QV3}. From this and the above isomorphism of Lie algebras the result now follows directly. It is also important to note that the MASA's of $\uni(p,1)$ are of the form $\h_1 \times \R$ where $\h_1$ is a MASA of $\su(p,1)$.
\end{proof}

After exponentiating the even MASA's listed above we obtain the following even MASG's viewed through their actions on either $\mathbb{B}^{p|q}$ or $\mathbb{U}^{p|q}$. The content of Theorem~\ref{thm:MASA} is that, up to conjugacy, these are the only MASG's of $\SU(p,1|q)$.

      \textbf{Quasi-elliptic} group of super biholomorphisms of the super unit ball $\mathbb{B}^{p|q}$ is isomorphic to
$\mathbb{T}^{p+q}$ with the following group action:
                $$(t,s) :(z, \theta)\in \mathbb{B}^{p|q}\mapsto (tz,s\theta) = (t_1z_1 ,\ldots, t_p z_p, s_{1} \theta_1,\ldots,s_{q}\theta_q )\in \mathbb{B}^{p|q},$$
for each $t = (t_1 ,\ldots, t_{p},s_1,\ldots,s_q )$.

Note that if the super function $F$ is invariant under the action of the  quasi-elliptic group, then
$$F(z,\xi)=\sum_{I\subset Q} f_I(r)\xi_I \xi_I^*,$$
where $r = (r_1 , ..., r_p ) = (|z_1 |, \ldots, |z_p |)$.

 \textbf{Quasi-parabolic} group of biholomorphisms of the super Siegel domain $\mathbb{U}^{p|q}$ is isomorphic to $\mathbb{T}^{p-1}\times\mathbb{R}\times \mathbb{T}^{q}$ with the following group action:
                     $$  (t, h,s) : (z' , z_p,\theta ) \in \mathbb{U}^{p|q} \mapsto (tz' , z_p + h,s\theta) \in \mathbb{U}^{p|q},$$
for each $(t, h,s) \in \mathbb{T}^{p-1}\times\mathbb{R} \times\mathbb{T}^{q}$.

In this case, if the super function $F$ is invariant under the action of the  quasi-parabolic group, then
$$F(z,\xi)=\sum_{I\subset Q} f_I(r', \text{Im }(z_n))\xi_I \xi_I^*,$$
where $r' = (r_1 , ..., r_{p-1} ) = (|z_1 |, \ldots, |z_{p-1} |)$.

\textbf{Quasi-hyperbolic} group of biholomorphisms of the Siegel domain $\mathbb{U}^{p|q}$ is isomorphic to $\mathbb{T}^{p-1}\times\mathbb{R}_+\times \mathbb{T}^{q}$ with the following group action:
                     $$ (t, r,s) : (z' , z_p,\theta ) \in \mathbb{U}^{p|q} \mapsto(r^{1/2} t z' , r z_p, r^{1/2} s \theta ) \in \mathbb{U}^{p|q},$$
for each $(t, r,s) \in\mathbb{T}^{p-1}\times\mathbb{R}_+\times \mathbb{T}^{q}$.

We now have that, if the super function $F$ is invariant under the action of the  quasi-hyperbolic group, then
$$F(z,\xi)=\sum_{I\subset Q} f_I\left(\rho_1,\ldots,\rho_{n-1} , \arg(z_n - i|z'|^2)\right) |z_n - i|z'|^2|^{-|I|}  \xi_I \xi_I^*,$$
where  $z'=(z_{1},...,z_{p-1})$ and $$\rho_k=\frac{|z_k|}{\sqrt{|z'|^2 + |z_n - i|z'|^2|}},$$ for $k=1,\ldots,n-1.$

\textbf{Nilpotent} group of biholomorphisms of the Siegel domain $\mathbb{U}^{p|q}$ is isomorphic
to $\mathbb{R}^{p-1}\times\mathbb{R}\times\mathbb{T}^{q}$ with the following group action:
           $$(b, h,s) : (z' , z_p,\theta )\in  \mathbb{U}^{p|q} \mapsto (z' + b, z_p + h + 2iz' \cdot b + i|b|^2 , s\theta)\in \mathbb{U}^{p|q},$$
for each $(b, h,s)\in \mathbb{R}^{p-1}\times\mathbb{R}\times \mathbb{T}^{q}$.

In this case, we have that if the super function $F$ is invariant under the action of the  nilpotent group, then
$$F(z,\xi)=\sum_{I\subset Q} f_I(\Ima z',\Ima z_n - |z'|^2) \xi_I \xi_I^*,$$
where $ z'=(z_1 , ..., z_{p-1} )$.

\textbf{Quasi-nilpotent} group of biholomorphisms of the Siegel domain $\mathbb{U}^{p|q}$ is isomorphic to $\mathbb{T}^k\times\mathbb{R}^{p-k-1}\times\mathbb{R}\times \mathbb{T}^{q}$ where  $0 < k < p - 1$, with the following group action:
$$(t,b, h,s) : (z',z'' , z_p,\theta )\in  \mathbb{U}^{p|q} \mapsto (tz',z'' + b, z_p + h + 2iz' \cdot b + i|b|^2 , s\theta)\in \mathbb{U}^{p|q} ,$$
for each $(t, b, h,s) \in\mathbb{T}^k\times\mathbb{R}^{p-k-1}\times\mathbb{R}\times \mathbb{T}^{q}$.

And for this case, we have that if the super function $F$ is invariant under the action of the  quasi-nilpotent group, then
$$F(z,\xi)=\sum_{I\subset Q} f_I(r, y',\Ima z_n - |z'|^2) \xi_I \xi_I^*,$$
where   $r=(|z_1|,...,|z_k|)$, $y' = \Ima w'$, and $w'=(z_{k+1},...,z_{p-1})$.


\section{Super Bargmann Transform}
\label{sec:superBargmann}

In \cite{QV2} it was introduced a Bargmann type transform for each of the five cases on the unit ball.
These Bargmann transforms are used in \cite{QV2} to provide very useful descriptions of the Bergman spaces in terms of coordinates corresponding to the actions of maximal abelian subgroups.
In this section we define Bargmann type transforms corresponding to the even MASA's considered above. These transforms are natural analogues of those defined in \cite{QV2}

\subsection{Quasi-elliptic case}

Denote by  $\tau (\mathbb{B}^p )$ the base of the unit ball $\mathbb{B}^p$, considered as a Reinhardt
domain. In other words, we have
     $$ \tau (\mathbb{B}^p ) = \{r = (r_1 , ..., r_p ) = (|z_1 |, \ldots, |z_p |) : r^2 = r_1^2 + \cdots + r_n^2 \in [0, 1)\},$$
which is contained in $\mathbb{R}^p_+$.  Consider in $\mathbb{C}^p$ the polar coordinates
$z_k = t_k r_k$, where $r_k\in \mathbb{R}_+ $ and $t_k \in \mathbb{T}$, for $k = 1,\ldots, p$.
Then, with respect to the identification
                       $$ z = (z_1 ,\ldots, z_p ) = (t_1 r_1 , \ldots, t_p r_p ) = (t, r),$$
where $t = (t_1 , ..., t_p ) \in\mathbb{T}^p, r = (r_1 ,\ldots, r_p ) \in  \tau (\mathbb{B}^p )$, we have
$\mathbb{B}^p  = \mathbb{T}^p \times   \tau (\mathbb{B}^p )$, and
                      $$   L_2 ( \mathbb{B}^p, \mu_{\nu}) = L_2 (\mathbb{T}^p ) \otimes L_2 ( \tau(\mathbb{B}^p), \mu_{\nu}),$$
where
$$L_2 (\mathbb{T}^p )= \bigotimes_{k=1}^p L_2 \left(\mathbb{T},\frac{d t_k}{i t_k} \right),$$
and the measure $d\mu_{\nu}(r)$ in $L_2 ( \tau(\mathbb{B}^p), \mu_{\nu})$ is given by
    $$d\mu_{\nu}(r)=\frac{\Gamma(\nu)}{\pi^p\Gamma(\nu-p)}(1-r^2)^{\nu-p-1}r dr.$$
We denote the discrete Fourier transform $ \mathcal{F} : L_2 (\mathbb{T} )\mapsto  l_2 = l_2 (\mathbb{Z})$ by
$$\mathcal{F}: f\mapsto c_n= \frac{1}{\sqrt{2\pi}}\int_{\mathbb{T}} f(t)t^{-n} \frac{dt}{it}.$$
Of course, the operator $\mathcal{F}$ is unitary and
$$\mathcal{F}^{-1}=\mathcal{F}^*: \{ c_n\}_{n\in \mathbb{Z}}\mapsto f=\frac{1}{\sqrt{2\pi}} \sum_{n\in\mathbb{Z}} c_n t^n.$$
We consider the operator
     $$U = \mathcal{F}_{(p)}\otimes  I : L_2 (\mathbb{T}^p ) \otimes L_2 ( \mathbb{B}^p, \mu_{\nu})\mapsto l_2 (\mathbb{Z}^p ) \otimes L_2 ( \mathbb{B}^p, \mu_{\nu}),$$
where $\mathcal{F}_{(p)}=\mathcal{F}\otimes\cdots \otimes \mathcal{F}$.
As in \cite{QV2}, we use the isometric embedding
$$R_{0,\nu} : l_2 (\mathbb{Z}^p_+ )\longrightarrow  l_2 (\mathbb{Z}^p ) \otimes L_2 ( \mathbb{B}^p, \mu_{\nu}),$$
defined by
$$R_{0,\nu}:\{ c_n \}_{n\in \mathbb{Z}^p_+}\mapsto c_n(r)= \left\{
\begin{array}{ll}
\left(\frac{(2\pi)^p \Gamma(|n|+\nu)}{n! \Gamma(\nu)}\right)^{1/2} c_n r^n, & \text{ for } n\in \mathbb{Z}^p_+\\
 0, & \text{ for } n\in \mathbb{Z}^p \setminus \mathbb{Z}^p_+
 \end{array}. \right.$$
Hence, it is easily seen that the map
 $$R_{0,\nu}^* :  l_2 (\mathbb{Z}^p ) \otimes L_2 ( \mathbb{B}^p, \mu_{\nu}) \longrightarrow  l_2 (\mathbb{Z}^p_+ ) $$
satisfies
$$R^*_{0,\nu}:\{ f_n(r) \}_{n\in \mathbb{Z}^p_+}\mapsto  \left\{
\left(\frac{(2\pi)^p \Gamma(|n|+\nu)}{n! \Gamma(\nu)}\right)^{1/2} \int_{ \tau(\mathbb{B}^p)} f_n(r) r^n d\mu_{\nu}  \right\}_{\mathbb{Z}^p_+}.$$

In \cite{QV2}  the authors introduced
the operator $R_{\nu} = R^*_{0,\nu} U$  from $L_2 ( \mathbb{B}^p, \mu_{\nu})$ onto $l_2 (\mathbb{Z}^p_+ )$ and the adjoint operator $ R^*_{\nu} = U^*  R_{0,\nu}$
 from $ l_2 (\mathbb{Z}^p_+ )$ onto  $H^2_{\nu}(\mathbb{B}^p )$.  They proved that $ R^*_{\nu}$ is the isometric isomorphism from $l_2 (\mathbb{Z}^p_+ )$ onto the subspace $H^2_{\nu}(\mathbb{B}^p )$.

 Furthermore, we have
\begin{align}\label{BT-QE}
 R_{\nu}R^*_{\nu}=I & : l_2 (\mathbb{Z}^p_+)\rightarrow l_2 (\mathbb{Z}^p_+ ),\\ \nonumber
R^*_{\nu}R_{\nu}=P_{\mathbb{B}^{p },\nu}&: L_2 ( \mathbb{B}^p, \mu_{\nu}) \rightarrow H^2_{\nu}(\mathbb{B}^p ),
\end{align}
where $P_{\mathbb{B}^{p },\nu}$ is the Bergman projection.

The inner product in $(l_2(\mathbb{Z}_+^p))^{2^q}$ is given by
\begin{align*}
    &\langle \{  a_{M,n}  \}_{n\in \mathbb{Z}^p_+})_{M\subset Q},(\{  b_{M,n}  \}_{n\in \mathbb{Z}^p_+})_{M\subset Q} \rangle \\
    &=  \sum_{M\subset Q} \frac{\Gamma(\nu)}{\Gamma(\nu+|M|)} \left(\sum_{n\in \mathbb{Z}^p_+} a_{M,n} \overline{b_{M,n}}\right).
\end{align*}

\begin{defn}\label{SBT-QE}
 Consider the operator
$$R_{\nu,(p|q)}: H^2_{\nu}(\mathbb{B}^{p|q })\rightarrow (l_2(\mathbb{Z}_+^p))^{2^q},$$
defined by $$R_{\nu,(p|q)}\left( \sum_{M\subset Q} \psi_M \xi_M \right)=
\left(   R_{\nu+|M|}(\psi_M)\right)_{M\subset Q},$$ whose explicit expression is given by
\begin{align*}
 &R_{\nu,(p|q)}\left( \sum_{M\subset Q} \psi_M \xi_M \right)\\
&=
\left( \left\{  (2\pi)^{-\frac{p}{2}}
\left(      \frac{(2\pi)^p \Gamma(|n|+\nu+|M|)}{n! \Gamma(\nu+|M|)}\right)^{1/2}
\int_{ \mathbb{B}^p} \psi_M \overline{z^n} d\mu_{\nu+|M|}(z)  \right\}_{\mathbb{Z}^p_+} \right)_{M\subset Q}.
\end{align*}
The adjoint operator
$$R_{\nu,(p|q)}^*: (l_2(\mathbb{Z}_+^p))^{2^q}\rightarrow H^2_{\nu}(\mathbb{B}^{p|q })$$
is defined by
$$R_{\nu,(p|q)}^* ((\{  c_{M,n}  \}_{n\in \mathbb{Z}^p_+})_{M\subset Q})= \sum_{M\subset Q} R^*_{\nu+|M|}(\{  c_{M,n}  \}_{n\in \mathbb{Z}^p_+}) \xi_M, $$
whose explicit expression is now
\begin{align*}
 &R_{\nu,(p|q)}^*((\{  c_{M,n}  \}_{n\in \mathbb{Z}^p_+})_{M\subset Q})\\
&=\sum_{M\subset Q} \left((2\pi)^{-\frac{p}{2}} \sum_{n\in \mathbb{Z}^p_+} \left(      \frac{(2\pi)^p \Gamma(|n|+\nu+|M|)}{n! \Gamma(\nu)}\right)^{1/2} c_{M,n} z^n \ \  \xi_M \right).
 \end{align*}
\end{defn}

\begin{thm}\label{TCS-QE}
The operators  $$R_{\nu,(p|q)}: H^2_{\nu}(\mathbb{B}^{p|q })\rightarrow (l_2(\mathbb{Z}_+^p))^{2^q}$$ and
$$R_{\nu,(p|q)}^*: (l_2(\mathbb{Z}_+^p))^{2^q}\rightarrow H^2_{\nu}(\mathbb{B}^{p|q })$$
are isometric isomorphisms.
Furthermore, we have
\begin{align*}
 R_{\nu,(p|q)}R^*_{\nu,(p|q)}=I & : (l_2(\mathbb{Z}_+^p))^{2^q}\rightarrow (l_2(\mathbb{Z}_+^p))^{2^q},\\
R^*_{\nu,(p|q)}R_{\nu,(p|q)}=I &: H^2_{\nu}(\mathbb{B}^{p|q} ) \rightarrow H^2_{\nu}(\mathbb{B}^{p|q} ).
\end{align*}
\end{thm}

\begin{proof}
First, we will prove that  $R_{\nu,(p|q)}$ is a unitary operator.
We first take an element $(\{  c_{M,n}  \}_{n\in \mathbb{Z}^p_+})_{M\subset Q})$  in $(l_2(\mathbb{Z}_+^p))^{2^q}$, then
\begin{align*}
I&=  \left\|R_{\nu,(p|q)}^*((\{  c_{M,n}  \}_{n\in \mathbb{Z}^p_+})_{M\subset Q}) \right\|^2_{\mathbb{B}^{p|q},\nu}\\
&= \left( \sum_{M\subset Q} R_{\nu+|M|}(\{  c_{M,n}  \}_{n\in \mathbb{Z}^p_+}) \xi_M,\sum_{M\subset Q} R_{\nu+|M|}(\{  c_{M,n}  \}_{n\in \mathbb{Z}^p_+}) \xi_M    \right)_{\mathbb{B}^{p|q},\nu}\\
&=\sum_{m=1}^q    \frac{\Gamma(\nu)}{\Gamma(\nu+m)}   \sum_{M\subset Q, |M|=m}  \| R_{\nu+|M|}(\{  c_{M,n}  \}_{n\in \mathbb{Z}^p_+}) \|^2_{\mathbb{U}^{p},\nu+m}.
\end{align*}
By (\ref{BT-QE}) we know that $R_{\nu}$ is an isometric isomorphism, then
\[
I= \sum_{M\subset Q} \frac{\Gamma(\nu)}{\Gamma(\nu+|M|)}  \sum_{n\in \mathbb{Z}^p_+} |c_{M,n} |^2=
\left\|(\{  c_{M,n}  \}_{n\in \mathbb{Z}^p_+})_{M\subset Q}\right\|^2_{(l_2(\mathbb{Z}_+^p))^{2^q}},
\]
which proves our claim.
\end{proof}

\subsection{Quasi-parabolic case}

We define  $\mathbb{D}^p = \mathbb{C}^{p-1} \times\mathbb{R}\times \mathbb{R}_+$ whose points we denote by $(z',u,v)$.
Consider the space $L_2(\mathbb{D}^p,\eta_{\nu})$, where $\eta_\nu$ is given by
$$\eta_\nu(z',u,v)=\frac{\Gamma(\nu)}{4\pi^p \Gamma(\nu-p)} v^{\nu-p-1}.$$
Consider the operator $U_0:L_2(\mathbb{B}^p,\mu_{\nu})\longrightarrow L_2(\mathbb{D}^p,\eta_{\nu})$ defined by
$$U_0(f)(z',u,v)=f(\kappa(z',u,v)),$$
where $\kappa(z' , u, v) = (z' , u + iv + i|w'|)$.
It is clear that the operator is unitary and the inverse operator is given by
$$U_0^{-1}(f)(z)=f(\kappa^{-1}(z)).$$

We represent the space $L_2 (\mathbb{D}^p, \eta_\nu )$ as the following tensor product
$$L_2 (\mathbb{D}^p, \eta_\nu ) = L_2 (\mathbb{C}^{p-1} ) \otimes L_2 (\mathbb{R})\otimes L_2 (\mathbb{R}_+ , \eta_\nu ),$$
and consider the unitary operator $U_1 = I \otimes F \otimes I$ acting on it. Here
$F$ is the standard Fourier transform on $L_2 (\mathbb{R})$ and $f(z',u,v)\mapsto U_1f(z',\xi,v)$.

On the other hand, we have the decomposition $$L_2 (\mathbb{D}^p, \eta_\nu ) = L_2 (\mathbb{R}_+^{p-1} ) \otimes L_2 (\mathbb{T}^{p-1} ) \otimes L_2 (\mathbb{R})
\otimes L_2 (\mathbb{R}_+ , \eta_\nu ),$$
where
$$L_2 (\mathbb{T}^{p-1} )= \bigotimes_{k=1}^{p-1}L_2 \left(\mathbb{T},\frac{d t_k}{i t_k}  \right).$$
As before, we consider the discrete Fourier transform $ \mathcal{F} : L_2 (\mathbb{T}^p )\mapsto  l_2 = l_2 (\mathbb{Z})$, and consider the unitary operator $U_2 = I \otimes \mathcal{F}_{(p-1)} \otimes I\otimes I$ acting from
 $$L_2 (\mathbb{R}_+^{p-1} ) \otimes L_2 (\mathbb{T}^{p-1} ) \otimes L_2 (\mathbb{R})\otimes L_2 (\mathbb{R}_+ , \eta_\nu )$$
onto
 $$L_2 (\mathbb{R}_+^{p-1} ) \otimes l_2 (\mathbb{Z}^{p-1} ) \otimes L_2 (\mathbb{R})\otimes L_2 (\mathbb{R}_+ , \eta_\nu )$$
$$=l_2 (\mathbb{Z}^{p-1}, L_2 (\mathbb{R}_+^{p-1} )   \otimes L_2 (\mathbb{R})\otimes L_2 (\mathbb{R}_+ , \eta_\nu ),$$
where $\mathcal{F}_{(p-1)} = \mathcal{F} \otimes \cdots\otimes\mathcal{F}$ and $
f(z',\xi,v)\mapsto \{U_2 f(r,n,\xi,v)\}_{n\in \mathbb{Z}^{p-1}}$

We now the consider isometric embedding
$$R_{\nu,0} :l_2 (\mathbb{Z}^{p-1}_+, L_2 (\mathbb{R}_+ ))\rightarrow
 l_2 (\mathbb{Z}^{p-1}, L_2 (\mathbb{R}_+^{p-1} )   \otimes L_2 (\mathbb{R})\otimes L_2 (\mathbb{R}_+ , \eta_\nu ))$$
defined by the assignment
\begin{align*}
 &R_{\nu,0} \{b_n (\xi)\}_{n\in \mathbb{Z}^{p-1}_+}\\
 &= \left\{ \chi_{\mathbb{Z}^{p-1}_+}(n) \chi_{\mathbb{R}_+(\xi)  }
\left(
 \frac{ 2(2\pi)^{p}(2 \xi)^{|n|+\nu-1 }}{n! \Gamma(\nu)}\right)^{\frac{1}{2}}
r^n e^{−\xi(|z'|^2+v)}b_n (\xi)
\right\}_{n\in \mathbb{Z}^{p-1}_+},
\end{align*}
where the function $b_n (\xi)$ is extended by zero for $\xi \in \mathbb{R} \setminus \mathbb{R}_+$ and $n\in\mathbb{Z}^{p-1} \setminus\mathbb{Z}^{p-1}_+ $.
Its adjoint
$$R_{\nu,0}^* : l_2 (\mathbb{Z}^{p-1}, L_2 (\mathbb{R}_+^{p-1} )
\otimes L_2 (\mathbb{R})\otimes L_2 (\mathbb{R}_+ , \eta_\nu ))\rightarrow l_2 (\mathbb{Z}^{p-1}_+, L_2 (\mathbb{R}_+ ))$$
is given by the expression
\begin{align*}
& R_{\nu,0}^*: \{d_n (r,\xi,v)\}_{n\in \mathbb{Z}^{p-1}} = \\
&\left\{  \chi_{\mathbb{R}_+(\xi)  }
\left(
 \frac{ 2(2\pi)^{p}(2 \xi)^{|n|+\nu-1 }}{n! \Gamma(\nu)}\right)^{\frac{1}{2}}
\int_{\mathbb{R}_+^n}
r^n e^{−\xi(|z'|^2+v)}d_n (r,\xi,v)rdr \frac{c_{\nu}v^{\nu-p-1}}{4}dv
\right\}_{n\in \mathbb{Z}^{p-1}_+}.
\end{align*}

In \cite{QV2}  the authors introduced the  operator
 $R_{\nu} = R^*_{0,\nu} U$ from $L_2 (\mathbb{U}^{p} , \mu_\nu )$ onto $l_2(\mathbb{Z}_+^{p-1}, L_2(\mathbb{R}_+))$,
 and the adjoint operator $ R^*_{\nu} = U^*  R_{0,\nu}$ from
  $l_2(\mathbb{Z}_+^{p-1}, L_2(\mathbb{R}_+))$ onto
  $H^2_{\nu}(\mathbb{U}^p )$ where $U=U_2U_1U_0$. They proved that $ R^*_{\nu}$ is the isometric isomorphism of $l_2(\mathbb{Z}_+^{p-1}, L_2(\mathbb{R}_+))$ onto the subspace $H^2_{\nu}(\mathbb{U}^p )$.
Furthermore, we have
\begin{align*}
 R_{\nu}R^*_{\nu}=I & : l_2(\mathbb{Z}_+^{p-1}, L_2(\mathbb{R}_+))\rightarrow l_2(\mathbb{Z}_+^{p-1}, L_2(\mathbb{R}_+))\\
R^*_{\nu}R_{\nu}=P_{\mathbb{U}^{p },\nu}&: L_2 ( \mathbb{U}^p, \mu_{\nu}) \rightarrow H^2_{\nu}(\mathbb{U}^p ),
\end{align*}
where $P_{\mathbb{U}^{p },\nu}$ is the Bergman projection of $L_2 ( \mathbb{U}^p, \mu_{\nu})$ to $  H^2_{\nu}(\mathbb{U}^p ).$

The inner product in $(l_2(\mathbb{Z}_+^{p-1}, L_2(\mathbb{R}_+)) )^{2^q}$ is given by
\begin{align*}
&  \langle( \{  a_{M,n}(\xi)  \}_{n\in \mathbb{Z}^{p-1}_+})_{M\subset Q},(\{  b_{M,n} (\xi) \}_{n\in \mathbb{Z}^{p-1}_+})_{M\subset Q} \rangle\\
&=\sum_{M\subset Q} \frac{\Gamma(\nu)}{\Gamma(\nu+|M|)} \left(\sum_{n\in \mathbb{Z}^{p-1}_+} ((a_{M,n}(\xi),b_{M,n}(\xi))_{L_2(\mathbb{R}_+)} \right).
\end{align*}

\begin{defn}\label{SBT-QP}
 Consider the operator
$$R_{\nu,(p|q)}: H^2_{\nu}(\mathbb{U}^{p|q })\rightarrow (l_2(\mathbb{Z}_+^{p-1}, L_2(\mathbb{R}_+)) )^{2^q}$$
defined by $$R_{\nu,(p|q)}\left( \sum_{M\subset Q} \psi_M \xi_M \right)=
\left(   R_{\nu+|M|}(\psi_M)\right)_{M\subset Q},$$
whose explicit expression is given by
\begin{align*}
 & R_{\nu,(p|q)}\left( \sum_{M\subset Q} \psi_M \xi_M \right)\\
&=\left(
\left\{
(2\pi)^{-\frac{p}{2}} \left(
 \frac{2(2\pi)^{p} (2 \xi)^{|n|+\nu+|M|-1 }}{n! \Gamma(\nu+|M|)}\right)^{\frac{1}{2}}
\right.
\right.\\
&\times  \left. \left.
\int_{\mathbb{U}^p}\psi_M(z) (\bar{z'})^n e^{-i\xi z_n}(\text{Im}z_n-|z'|)^{\nu+|M|-p-1} dv(z)
\right\}_{\mathbb{Z}^{p-1}_+}
\right)_{M\subset Q}.
\end{align*}
Its adjoint operator is
$$R_{\nu,(p|q)}^*: (l_2(\mathbb{Z}_+^{p-1}, L_2(\mathbb{R}_+)) )^{2^q}\rightarrow H^2_{\nu}(\mathbb{U}^{p|q })$$
defined by $$R_{\nu,(p|q)}^* ((\{  b_{M,n}(\xi)  \}_{n\in \mathbb{Z}^{p-1}_+})_{M\subset Q})= \sum_{M\subset Q} R_{\nu+|M|}(\{  b_{M,n}(\xi)  \}_{n\in \mathbb{Z}^{p-1}_+}) \xi_M, $$
whose explicit expression is now given by
\begin{align*}
 & R_{\nu,(p|q)}^*((\{  b_{M,n}(\xi)  \}_{n\in \mathbb{Z}^{p-1}_+})_{M\subset Q})\\
&=\sum_{M\subset Q}
 \left(
 (2\pi)^{-\frac{p}{2}}
\sum_{n\in\mathbb{Z}^{p-1}_+}
\int_{\mathbb{R}_+}  \left(
 \frac{2(2\pi)^{p} (2 \xi)^{|n|+\nu+|M|-1 }}{n! \Gamma(\nu+|M|)}\right)^{\frac{1}{2}} b_{n,M}(\xi) (z')^n e^{i\xi z_n}d\xi
 \right) \xi_M.
\end{align*}

\end{defn}

\begin{thm}
The operators  $$R_{\nu,(p|q)}: H^2_{\nu}(\mathbb{U}^{p|q })\rightarrow (l_2(\mathbb{Z}_+^{p-1}, L_2(\mathbb{R}_+)) )^{2^q}$$ \\ and
 $$R_{\nu,(p|q)}^*: (l_2(\mathbb{Z}_+^{p-1}, L_2(\mathbb{R}_+)) )^{2^q}\rightarrow H^2_{\nu}(\mathbb{U}^{p|q })$$
are isometric isomorphisms.
Furthermore, we have
\begin{align*}
 R_{\nu,(p|q)}R^*_{\nu,(p|q)}=I & : (l_2(\mathbb{Z}_+^{p-1}, L_2(\mathbb{R}_+)) )^{2^q}\rightarrow (l_2(\mathbb{Z}_+^{p-1}, L_2(\mathbb{R}_+)) )^{2^q}\\
R^*_{\nu,(p|q)}R_{\nu,(p|q)}=I &: H^2_{\nu}(\mathbb{U}^{p|q} ) \rightarrow H^2_{\nu}(\mathbb{U}^{p|q} ).
\end{align*}

\end{thm}
The proof of this result is similar to that of Theorem \ref{TCS-QE}.


\subsection{Nilpotent case}

We consider the space
\begin{equation*}
L_2(\mathbb{D}^p,\eta_{\nu}) = L_2(\C^{p-1}) \otimes L_2(\R) \otimes L_2(\R_+,\eta_{\nu}),
\end{equation*}
and the unitary operator $U_1=I \otimes F \otimes I$ acting on it, where $F$ is the Fourier transform on $L_2(\R)$.

Using the standard Cartesian coordinates $x'=(x_1,...,x_{p-1})$ and\\ $y'=(y_1,..., y_{p-1})$, where $z_k=x_k+iy_k$, in $\C^{p-1}=\R^{p-1} \times \R^{p-1}$,
we have
\begin{equation*}
L_2(\mathbb{D}^p,\eta_{\nu}) = L_2(\R^{p-1}) \otimes L_2(\R^{p-1}) \otimes L_2(\R) \otimes L_2(\R_+,\eta_{\nu}).
\end{equation*}
Consider the unitary operator $U_2= F_{(p-1)}\otimes I \otimes I \otimes I$, where \\ $F_{(p-1)} = F_1 \otimes ... \otimes F_{p-1}$ is $(p-1)$-dimensional Fourier transform, acting on this tensor decomposition.

Consider the change of variables
\begin{equation*}
u_k = \frac{1}{2\sqrt{\xi}} \xi_k - \sqrt{\xi} y_k , \ \ \ \
v_k = \frac{1}{2\sqrt{\xi}} \xi_k + \sqrt{\xi} y_k, \ \ \ \, k = 1,..., p-1,
\end{equation*}
which is equivalent to
\begin{equation*}
\xi_k = \sqrt{\xi} \left(u_k + v_k\right), \ \ \ \
y_k = \frac{1}{2\sqrt{\xi}} \left(-u_k + v_k\right), \ \ \ \, k = 1,..., p-1,
\end{equation*}
and the corresponding unitary operator $U_3$ acting on $L_2(\R^{p-1}) \otimes L_2(\R^{p-1}) \otimes L_2(\R) \otimes L_2(\R_+,\eta_{\nu})$
by the rule

\begin{equation*}
(U_3 \varphi)(u',v',\xi,v) = \varphi \left(\sqrt{\xi} \left(u' + v'\right), \frac{1}{2\sqrt{\xi}} \left(-u' + v'\right),\xi, v\right),
\end{equation*}
where $u'=(u_1,...,u_{p-1})$ and $v'=(v_1,...,v_{p-1})$.

In this case, we consider the isometric embedding
\begin{equation*}
R_{0,\nu}: L_2(\R^{p-1} \times \R_+) \longrightarrow L_2(\R^{p-1}) \otimes L_2(\R^{p-1}) \otimes L_2(\R) \otimes L_2(\R_+,\eta_{\nu})
\end{equation*}
given by the assignment
\begin{equation*}
R_{0,\nu} ( \psi(u',\xi))=
\pi^{-\frac{p-1}{4}}\,  e^{-\xi v -\frac{|v'|^2}{2}}\, \chi_{\R_+}(\xi) \left(\frac{4(2\xi)^{\nu-p}} {c_{\nu}\Gamma(\nu-p)}\right)^{\frac{1}{2}} \psi(u',\xi),
\end{equation*}
where the function $\psi(u',\xi)$ is extended by zero for $\xi \in \R \setminus \R_+$ for each $u' \in \R^{p-1}$. The adjoint operator
\begin{equation*}
R_{0,\nu}^* : L_2(\R^{p-1}) \otimes L_2(\R^{p-1}) \otimes L_2(\R) \otimes L_2(\R_+,\eta_{\nu}) \longrightarrow L_2(\R^{p-1} \times \R_+)
\end{equation*}
obviously has the form
\begin{align*}
&R_{0,\nu}^* ( \varphi(u',v',\xi,v))\\
&=\pi^{-\frac{p-1}{4}} \int_{\R^{p-1} \times \R_+} e^{-\xi v -\frac{|v'|^2}{2}} \left(\frac{4(2\xi)^{\nu-p}} {c_{\nu}\Gamma(\nu-p)}\right)^{\frac{1}{2}} f(u',v',\xi,v)\,dv'\, \frac{c_{\nu}}{4}\,v^{\nu-p-1}dv.
\end{align*}

In \cite{QV2}  the authors introduced the  operator $R_{\nu}=R_{0,\nu}^* U$ from $L_2(\mathbb{D}^p,\widetilde{\mu}_{\nu})$ onto $L_2(\R^{p-1} \times \R_+)$ and
the adjoint operator $R_{\nu}^*$ from $L_2(\R^{p-1} \times \R_+)$ onto the subspace
$\Aa^2_{\nu}(\mathbb{D}^p)$. They proved that $ R^*_{\nu}$ is the isometric isomorphism of $L_2(\R^{p-1} \times \R_+)$ onto the subspace $H^2_{\nu}(\mathbb{U}^p )$.
Furthermore
\begin{align*}
  R_{\nu} R_{\nu}^* = I &: L_2(\R^{p-1} \times \R_+) \longrightarrow L_2(\R^{p-1} \times \R_+), \\
  R_{\nu}^* R_{\nu} =B_{\mathbb{D}^p,\nu} &: L_2(\mathbb{D}^p,\widetilde{\mu}_{\nu}) \longrightarrow \Aa^2_{\nu}(\mathbb{D}^p),
\end{align*}
where $B_{\mathbb{D}^p,\nu}$ is the Bergman projection.

Now we define the analogue of the Bargmann transform for the super case.
\begin{defn}\label{SBT-N}
Consider the operator
$$R_{\nu,(p|q)}: H^2_{\nu}(\mathbb{U}^{p|q })\rightarrow (L_2(\R^{p-1} \times \R_+))^{2^q}$$
defined by $$R_{\nu,(p|q)}\left( \sum_{M\subset Q} \psi_M \xi_M \right)=
\left(   R_{\nu+|M|}(\psi_M)\right)_{M\subset Q}.$$ Whose adjoint operator
$$R_{\nu,(p|q)}^*: ( L_2(\R^{p-1} \times \R_+) )^{2^q}\rightarrow H^2_{\nu}(\mathbb{U}^{p|q })$$
is defined by $$R_{\nu,(p|q)}^* ( \psi_M (u',\xi) )_{M\subset Q})= \sum_{M\subset Q} R_{\nu+|M|}(\psi_M (u',\xi)) \xi_M .$$
  The inner product in $(L_2(\R^{p-1} \times \R_+) )^{2^q}$ is given by
\begin{align*}
   &\langle( \psi_M (u',\xi))_{M\subset Q}, (\phi_M (u',\xi))_{M\subset Q} \rangle \\
&=\sum_{M\subset Q} \frac{\Gamma(\nu)}{\Gamma(\nu+|M|)}  ( \psi_M (u',\xi) , \phi_M (u',\xi)  )_{L_2(\mathbb{R}_+)}.
\end{align*}

\end{defn}

\begin{thm}
The operators  $$R_{\nu,(p|q)}: H^2_{\nu}(\mathbb{U}^{p|q })\rightarrow (L_2(\R^{p-1} \times \R_+) )^{2^q}$$ \\ and
 $$R_{\nu,(p|q)}^*: (L_2(\R^{p-1} \times \R_+) )^{2^q}\rightarrow H^2_{\nu}(\mathbb{U}^{p|q })$$
are isometric isomorphisms.
Furthermore
\begin{align*}
 R_{\nu,(p|q)}R^*_{\nu,(p|q)}=I & : (L_2(\R^{p-1} \times \R_+) )^{2^q}\rightarrow (L_2(\R^{p-1} \times \R_+) )^{2^q}\\
R^*_{\nu,(p|q)}R_{\nu,(p|q)}=I &: H^2_{\nu}(\mathbb{U}^{p|q} ) \rightarrow H^2_{\nu}(\mathbb{U}^{p|q} ).
\end{align*}

\end{thm}
The proof is similar to that of theorem \ref{TCS-QE}.

\subsection{Quasi-nilpotent case}

Given an integer $1 \leq k \leq n-2$, we will write  points of $\mathbb{D}^p$ as $z=(z',w',z_n)$, where $z' \in \C^k$ and $w' \in \C^{p-k-1}$, and points of $\mathbb{D}^p$ as $(z',w',\zeta)$, respectively.

According to this notation we represent
\begin{equation*}
L_2(\mathbb{D}^p,\eta_{\nu})= L_2(\C^k)\otimes L_2(\C^{p-k-1}) \otimes L_2(\R) \otimes L_2(\R_+,\eta_{\nu}).
\end{equation*}
Applying, as in the previous two cases, the unitary operator $U_1= I \otimes I \otimes F \otimes I$,
we have that the image $\Aa_1(\mathbb{D}^p)=U_1(\Aa_0(\mathbb{D}^p))$.

Now introducing in $\C^k$ the polar coordinates, $z_l=r_l t_l$, where $r_l \in \R_+$, $t_l \in S^1=\T$, $l=1,\ldots,k$, and
Cartesian coordinates in $\C^{p-k-1}$, $x'=(x_1,...,x_{p-k-1})$, $y'=(y_1,...,y_{p-k-1})$, where $w_m = x_m +iy_m$, $m=1,...,p-k-1$, we have that $L_2(\mathbb{D}^p,\eta_{\nu})$ decomposes as the tensor product
\[
L_2(\R_+^k, rdr) \otimes L_2(\T^k) \otimes L_2(\R^{p-k-1}) \otimes L_2(\R^{p-k-1}) \otimes L_2(\R) \otimes L_2(\R_+,\eta_{\nu}).
\]
Consider the unitary operator $U_2= I \otimes \Fa_{(k)} \otimes F_{(p-k-1)} \otimes I \otimes I \otimes I$ acting from $L_2(\mathbb{D}^p,\eta_{\nu})$ onto
\begin{align*}
&  L_2(\R_+^k, rdr) \otimes l_2(\Z^{k} \otimes L_2(\R^{p-k-1}) \otimes L_2(\R^{p-k-1}) \otimes L_2(\R) \otimes L_2(\R_+,\eta_{\nu}) \\
&= l_2(\Z^k,\, L_2(\R_+^k, rdr) \otimes L_2(\R^{p-k-1}) \otimes L_2(\R^{p-k-1}) \otimes L_2(\R) \otimes L_2(\R_+,\eta_{\nu})),
\end{align*}
where  $\Fa_{(k)}= \Fa \otimes...\otimes \Fa$ is the $k$-dimensional discrete Fourier transform and $F_{(p-k-1)}= F \otimes...\otimes F$ is the $(p-k-1)$-dimensional Fourier transform.

Next consider the change of variables
\begin{equation*}
u_m = \frac{1}{2\sqrt{\xi}} \xi_m - \sqrt{\xi} y_m , \ \ \ \
v_m = \frac{1}{2\sqrt{\xi}} \xi_m + \sqrt{\xi} y_m, \ \ \ \, m = 1,..., p-k-1,
\end{equation*}
which is equivalent to
\begin{equation*}
\xi_m = \sqrt{\xi} \left(u_m + v_m\right), \ \ \ \
y_m = \frac{1}{2\sqrt{\xi}} \left(-u_m + v_m\right), \ \ \ \, m = 1,..., p-k-1.
\end{equation*}
Then, there is a corresponding unitary operator $U_3$ acting on
\begin{equation*}
l_2(\Z^k,\, L_2(\R_+^k, rdr) \otimes L_2(\R^{p-k-1}) \otimes L_2(\R^{p-k-1}) \otimes L_2(\R) \otimes L_2(\R_+,\eta_{\nu}))
\end{equation*}
by the assignment
\begin{equation*}
U_3  \{d_p(r,\xi',y',\xi,v)\}_{p\in \Z^k} = \left\{d_p \left(r, \sqrt{\xi} \left(u' + v'\right), \frac{1}{2\sqrt{\xi}} \left(-u' + v'\right),\xi, v\right) \right\}_{p\in \Z^k},
\end{equation*}
where $u'=(u_1,...,u_{p-k-1})$ and $v'=(v_1,...,v_{p-k-1})$.

In \cite{QV2} it is introduced the isometric imbedding $R_{0,\nu}$ from the Hilbert space $l_2(\Z^k_+,L_2(\R^{p-k-1} \times \R_+))$ into
\begin{equation*}
 l_2(\Z^k,(L_2(\R_+^k, rdr) \otimes L_2(\R^{p-k-1}) \otimes L_2(\R^{p-k-1}) \otimes L_2(\R) \otimes L_2(\R_+,\eta_{\nu})),
\end{equation*}
which maps $\{c_n(u',\xi)\}_{n\in \Z_+^k}$ to
\begin{equation*}
\left\{\pi^{-\frac{p-k-1}{4}}
\chi_{\Z^k_+}(n) \chi_{\R_+}(\xi)\left(
\frac{2^{k+2}(2\xi)^{|n|+\nu-p+k}}{c_{\nu}n!\Gamma(\nu-p)} \right)^{\frac{1}{2}} r^n e^{-\xi(|r|^2 +v)- \frac{|v'|^2}{2}}c_n (u',\xi)\right\}_{n \in \Z^k},
\end{equation*}
where the functions $c_n(u',\xi))$ are extended by zero for $\xi \in \R \setminus \R_+$ for each $u' \in \R^{p-k-1}$ and each $n \in \Z^k$.

The adjoint operator $R_{0,\nu}^*$ acts from
\begin{equation*}
l_2(\Z^k,(L_2(\R_+^k, rdr) \otimes L_2(\R^{p-k-1}) \otimes L_2(\R^{p-k-1}) \otimes L_2(\R) \otimes L_2(\R_+,\eta_{\nu}))
\end{equation*}
onto $l_2(\Z^k_+,L_2(\R^{p-k-1} \times \R_+))$ and maps a sequence $\{d_n(r,u',v',\xi,v)\}_{n \in \Z^k}$ into
\begin{align*}
& R_{0,\nu}^* ( \{d_n(r,u',v',\xi,v)\}_{n \in \Z^k})\\
=& \left\{ \pi^{-\frac{p-k-1}{4}} \left(\frac{2^{k+2}}{c_{\nu}}\,
\frac{(2\xi)^{|n|+\nu-p+k}}{n!\, \Gamma(\nu-p)} \right)^{\frac{1}{2}}\int_{\R^k_+ \times \R^{p-k-1} \times \R_+} r^n e^{-\xi(|r|^2 +v)- \frac{|v'|^2}{2}} \right. \\
&\times  \left. d_n(r,u',v',\xi,v)\, rdr\, dv'\, \frac{c_{\nu}v^{\nu-p-1}}{4}\,dv \right\}_{n \in \Z^k_+}.
\end{align*}

In \cite{QV2} the authors introduce the operator $R_{\nu}=R_{0,\nu}^* U$ from $L_2(\mathbb{D}^p,\widetilde{\mu}_{\nu})$ onto  $l_2(\Z^k_+,L_2(\R^{p-k-1} \times \R_+))$,
and the adjoint operator $R_{\nu}^*$ being the isometric isomorphism of $l_2(\Z^k_+,L_2(\R^{p-k-1} \times \R_+))$ onto the subspace
$\Aa^2_{\nu}(\mathbb{D}^p)$ of $L_2(\mathbb{D}^p,\widetilde{\mu}_{\nu})$.
Furthermore
\begin{align*}
  R_{\nu} R_{\nu}^* = I &: l_2(\Z^k_+,L_2(\R^{p-k-1} \times \R_+)) \longrightarrow l_2(\Z^k_+,L_2(\R^{p-k-1} \times \R_+)), \\
  R_{\nu}^* R_{\nu} =B_{\mathbb{D}^p,\nu} &: L_2(\mathbb{D}^p,\widetilde{\mu}_{\nu}) \longrightarrow \Aa^2_{\nu}(\mathbb{D}^p),
\end{align*}
where $B_{\mathbb{D}^p,\nu}$ is the Bergman projection.


\begin{defn}\label{SBT-QN}
 Consider the operator
$$R_{\nu,(p|q)}: H^2_{\nu}(\mathbb{U}^{p|q })\rightarrow ( l_2(\Z^k_+,L_2(\R^{p-k-1} \times \R_+)) )^{2^q}$$
defined by $$R_{\nu,(p|q)}\left( \sum_{M\subset Q} \psi_M \xi_M \right)=
\left(   R_{\nu+|M|}(\psi_M)\right)_{M\subset Q}.$$
The adjoint operator
$$R_{\nu,(p|q)}^*: (  l_2(\Z^k_+,L_2(\R^{p-k-1} \times \R_+)) )^{2^q}\rightarrow H^2_{\nu}(\mathbb{U}^{p|q })$$
is defined by $$R_{\nu,(p|q)}^* ( \{c_{M,n}(u',\xi)\}_{n\in \Z_+^k} )_{M\subset Q})= \sum_{M\subset Q} R_{\nu+|M|}(\{c_{M,n}(u',\xi)\}_{n\in \Z_+^k}) \xi_M ,$$
where
  the inner product in $(l_2(\Z^k_+,L_2(\R^{p-k-1} \times \R_+)) )^{2^q}$ is given by
\begin{align*}
&  \langle( \{a_{M,n}(u',\xi)\}_{n\in \Z_+^k} )_{M\subset Q}, (\{b_{M,n}(u',\xi)\}_{n\in \Z_+^k})_{M\subset Q} \rangle\\
&=\sum_{M\subset Q} \frac{\Gamma(\nu)}{\Gamma(\nu+|M|)}  \sum_{n\in \Z_+^k}  ( a_{M,n}(u',\xi) , b_{M,n}(u',\xi)  )_{L_2(\R^{p-k-1} \times \R_+)}.
\end{align*}

\end{defn}

\begin{thm}
The operators  $$R_{\nu,(p|q)}: H^2_{\nu}(\mathbb{U}^{p|q })\rightarrow (l_2(\Z^k_+,L_2(\R^{p-k-1} \times \R_+)) )^{2^q}$$ \\ and
 $$R_{\nu,(p|q)}^*: (l_2(\Z^k_+,L_2(\R^{p-k-1} \times \R_+)) )^{2^q}\rightarrow H^2_{\nu}(\mathbb{U}^{p|q })$$
are isometric isomorphisms.
Furthermore
\begin{align*}
 &R_{\nu,(p|q)}R^*_{\nu,(p|q)}=I  : \\
 &(l_2(\Z^k_+,L_2(\R^{p-k-1} \times \R_+)) )^{2^q}\rightarrow (l_2(\Z^k_+,L_2(\R^{p-k-1} \times \R_+)) )^{2^q}\\
&R^*_{\nu,(p|q)}R_{\nu,(p|q)}=I : \\
&H^2_{\nu}(\mathbb{U}^{p|q} ) \rightarrow H^2_{\nu}(\mathbb{U}^{p|q} ).
\end{align*}
\end{thm}
The proof is similar to that of theorem \ref{TCS-QE}.

\subsection{Quasi-hyperbolic case}

We represent $\mathbb{D}^p = \C^{p-1}\times \R \times \R_+$ in the form $\C^{p-1} \times \Pi$, where $\Pi$ is the upper half-plane, and introduce in $\mathbb{D}^p$ the ``non-isotropic'' upper semi-sphere
\begin{equation*}
\Omega = \{ (z', \zeta) \in \C^{p-1} \times \Pi : \
 |z'|^2 + |\zeta| = 1 \, \}.
\end{equation*}
The points of $\Omega$ admit the natural parameterization
\begin{align*}
z_k = s_k t_k, &  \ \mathrm{where} \ \
s_k \in [0,1), \ t_k \in S^1, \ \ k=1,..., p-1, \\
\zeta = \rho e^{i \theta}, &  \ \mathrm{where} \ \
\rho \in (0,1], \ \theta \in (0,\pi),
\end{align*}
and
\begin{equation*}
\sum_{k=1}^{p-1} s_k^2 + \rho = 1.
\end{equation*}
This in turn induces the following representation of the points $(z',\zeta) \in \mathbb{D}^p = \C^{p-1} \times \Pi$
\begin{equation*}
z_k = r^{\frac{1}{2}}s_k t_k, \ \ k=1,...,p-1,  \ \ \ \ \ \
\zeta = r\rho e^{i \theta},
\end{equation*}
where $r \in \R_+$.

We represent now $\mathbb{D}^p=\tau(\B^{p-1}) \times \T^{p-1} \times \R_+ \times (0,\pi)$, where
$$\tau(\B^{p-1})= \{ s=(s_1,...,s_{p-1}) \in \R_+^{p-1} : \sum_{k=1}^{p-1} s_k^2 < 1\}$$
is the base (in the sense of a Reinhardt domain) of the unit ball $\B^{p-1}$, and $\T^{p-1}= S^1 \times ... \times S^1$ is the $p-1$ dimensional torus.

Introduce the new coordinate system $(s,t,r,\theta)$ in $\mathbb{D}^p$, where we have $s=(s_1,...,s_{p-1}) \in \tau(\B^{p-1})$, $t=(t_1,...,t_{p-1}) \in \T^{p-1}$, $r \in \R_+$, and $\theta \in (0,\pi)$, which is connected with the old one $(z',\zeta)$ by the formulas
\begin{equation} \label{eq:s,t,r,theta}
s_k =\frac{|z_k|}{\sqrt{|z'|^2 + |\rho|}}, \ \ \
t_k = \frac{z_k}{|z_k|}, \ \ \ r=|z'|^2 + |\rho|, \ \ \
\theta = \arg \zeta,
\end{equation}
or
\begin{equation*}
z_k = r^{\frac{1}{2}}s_k t_k, \ \ \ \ \ \ \zeta = r(1-|s|^2)e^{i \theta},
\end{equation*}
where $k = 1,..., p-1$.

A direct computation shows that under the change of variables (\ref{eq:s,t,r,theta}) we have
\begin{equation*}
dv(z',\zeta)= r^p (1-|s|^2)\prod_{k=1}^{p-1}s_k ds_k \prod_{k=1}^{p-1} \frac{dt_k}{it_k}\,dr d\theta,
\end{equation*}
and
\begin{equation*}
\eta_{\nu}= \frac{c_{\nu}}{4}\, r^{\nu-p-1} (1-|s|^2)^{\nu-p-1} \frac{c_{\nu}}{4}\sin^{\nu-p-1} \theta.
\end{equation*}

Introduce the unitary operator $U_1 = I \otimes \Fa_{(p-1)} \otimes M_{\nu} \otimes I$ which acts from the space
\begin{align*}
& L_2(\tau(\B^{p-1}), (1-|s|^2)^{\nu-p}sds) \otimes L_2(\T^{p-1})\\ & \otimes L_2(\R_+, r^{\nu-1}dr) \otimes L_2((0,\pi),\frac{c_{\nu}}{4}\sin^{\nu-p-1} \theta d\theta)
\end{align*}
onto the space \begin{align*}
 l_2(\Z^{p-1}, L_2(\tau(\B^{p-1}), (1-|s|^2)^{\nu-p}sds) \otimes L_2(\R) \otimes L_2((0,\pi),\frac{c_{\nu}}{4}\sin^{\nu-p-1} \theta d\theta)),
\end{align*}
where the Mellin transform $M_{\nu}: L_2(\R_+, r^{\nu-1}dr) \longrightarrow L_2(\R)$ is given by
\begin{equation*}
(M_{\nu}\psi)(\xi) = \frac{1}{\sqrt{2\pi}} \int_{\R_+} r^{-i\xi + \frac{\nu}{2}-1}\, \psi(r)dr,
\end{equation*}
and the inverse is as follows
\begin{equation*}
 (M_{\nu}^{-1}\psi)(r) = \frac{1}{\sqrt{2\pi}} \int_{\R} r^{-i\xi - \frac{\nu}{2}}\, \psi(\xi)d\xi,
\end{equation*}
and $\Fa_{(p-1)} = \Fa \otimes ... \otimes \Fa$ is the $(p-1)$-dimensional discrete Fourier transform.

Introduce the isometric imbedding $R_{0,\nu}$ of the space $l_2(\Z^{p-1}_+,L_2(\R))$ into the space
\begin{equation*}
l_2(\Z_+^{p-1},L_2(\tau(\B^{p-1}), (1-|s|^2)^{\nu-p}sds) \otimes L_2(\R) \otimes L_2((0,\pi),\frac{c_{\nu}}{4}\sin^{\nu-p-1} \theta d\theta))
\end{equation*}
by the rule
\begin{equation*}
R_{0,\nu} : \{c_n(\xi)\}_{n \in \Z^{p-1}_+} \longmapsto
\{ c_n(\xi)\, \alpha_{n,\nu}(\xi)\,\beta_{n,\nu}(s,\xi,\theta)\}_{n\in \Z^{p-1}},
\end{equation*}
where the functions $\beta_{n,\nu}=\beta_{n,\nu}(s,\xi,\theta)$ and $\alpha_{n,\nu}(\xi)$ are given by
\begin{equation} \label{eq:beta_p}
\beta_{n,\nu} = s^n [1-(1+i)|s|^2]^{-\frac{\nu+|n|}{2} + i \xi} e^{-2\left(\xi +i \frac{\nu + |n|}{2}\right) \arctan\left[\left(1-i\frac{|s|^2}{1-|s|^2}\right)\tan \frac{\theta}{2} +\frac{|s|^2}{1-|s|^2}\right]}.
\end{equation}
and
\begin{equation}\label{eq:alpha_p}
\alpha_{n,\nu}(\xi)= \left(\int_{\tau(\B^{p-1})\times (0,\pi)}
|\beta_{n,\nu}(s,\xi,\theta)|^2 (1-|s|^2)^{\nu-p} \frac{c_{\nu}}{4}\sin^{\nu-p-1} \theta\, sds d\theta\right)^{-\frac{1}{2}}.
\end{equation}

The adjoint operator $R_{0,\nu}^*$ which acts from
\begin{equation*}
l_2(\Z^{p-1}_+,L_2(\tau(\B^{p-1}), (1-|s|^2)^{\nu-p}sds) \otimes L_2(\R) \otimes L_2((0,\pi),\frac{c_{\nu}}{4}\sin^{\nu-p-1} \theta d\theta))
\end{equation*}
onto the space $l_2(\Z^{p-1}_+,L_2(\R))$
has obviously the form
\begin{align*}
& R_{0,\nu}^* ( \{d_n(s,\xi,\theta)\}_{n \in \Z^{p-1}}) \\
=&  \left\{\alpha_{n,\nu}(\xi) \int_{\tau(\B^{p-1})\times (0,\pi)} \overline{\beta_{n,\nu}(s,\xi,\theta)}\, d_n(s,\xi,\theta)\,(1-|s|^2)^{\nu-p}\right.\\
&\times\left.
\frac{c_{\nu}}{4}\sin^{\nu-p-1} \theta\, sds d\theta \right\}_{n \in \Z^{p-1}_+}.
\end{align*}

In \cite{QV2} the authors introduced the operator $R_{\nu}=R_{0,\nu}^* U$ from $L_2(\mathbb{D}^p,\widetilde{\mu}_{\nu})$ onto $l_2(\Z^{p-1}_+,L_2(\R))$,
and the adjoint operator $R_{\nu}^*$ being an isometric isomorphism from $l_2(\Z^{p-1}_+,L_2(\R))$ onto the subspace
$\Aa^2_{\nu}(\mathbb{D}^p)$.
Furthermore
\begin{align*}
  R_{\nu} R_{\nu}^* = I &: l_2(\Z^{p-1}_+,L_2(\R)) \longrightarrow l_2(\Z^{p-1}_+,L_2(\R)), \\
  R_{\nu}^* R_{\nu} =B_{\mathbb{U}^p,\nu} &: L_2(\mathbb{U}^p,\widetilde{\mu}_{\nu}) \longrightarrow \Aa^2_{\nu}(\mathbb{U}^p),
\end{align*}
where $B_{\mathbb{U}^p,\nu}$ is the Bergman projection.


Now we define the analogue operator for the super case
\begin{defn}\label{SBT-QH}
 Consider the operator
$$R_{\nu,(p|q)}: H^2_{\nu}(\mathbb{U}^{p|q })\rightarrow (l_2(\Z^{p-1}_+,L_2(\R)) )^{2^q}$$
defined by $$R_{\nu,(p|q)}\left( \sum_{M\subset Q} \psi_M \xi_M \right)=
\left(   R_{\nu+|M|}(\psi_M)\right)_{M\subset Q}.$$
The adjoint operator
$$R_{\nu,(p|q)}^*: (  l_2(\Z^{p-1}_+,L_2(\R)) )^{2^q}\rightarrow H^2_{\nu}(\mathbb{U}^{p|q })$$
is defined by

$$R_{\nu,(p|q)}^* (\{c_{M,n}(\xi)\}_{n \in \Z^{p-1}_+} )_{M\subset Q})= \sum_{M\subset Q} R_{\nu+|M|}( \{c_{M,n}(\xi)\}_{n \in \Z^{p-1}_+} ) \xi_M ,$$
where
  the inner product in $(l_2(\Z^{p-1}_+,L_2(\R)) )^{2^q}$ is given by
\begin{align*}
&  \langle( \{a_{M,n}(\xi)\}_{n\in \Z_+^{p-1}} )_{M\subset Q}, (\{b_{M,n}(\xi)\}_{n\in \Z_+^{p-1}})_{M\subset Q} \rangle\\
&=\sum_{M\subset Q} \frac{\Gamma(\nu)}{\Gamma(\nu+|M|)}  \sum_{n\in \Z_+^{p-1}}  ( a_{M,n}(\xi) , b_{M,n}(\xi)  )_{L_2( \R_+)}.
\end{align*}

\end{defn}

\begin{thm}
The operators  $$R_{\nu,(p|q)}: H^2_{\nu}(\mathbb{U}^{p|q })\rightarrow (l_2(\Z^{p-1}_+,L_2(\R)) )^{2^q}$$ \\ and
 $$R_{\nu,(p|q)}^*: (l_2(\Z^{p-1}_+,L_2(\R)) )^{2^q}\rightarrow H^2_{\nu}(\mathbb{U}^{p|q })$$
are isometric isomorphisms.
Furthermore, we have
\begin{align*}
 R_{\nu,(p|q)}R^*_{\nu,(p|q)}=I & : (l_2(\Z^{p-1}_+,L_2(\R)))^{2^q}\rightarrow (l_2(\Z^{p-1}_+,L_2(\R)) )^{2^q}\\
R^*_{\nu,(p|q)}R_{\nu,(p|q)}=I &: H^2_{\nu}(\mathbb{U}^{p|q} ) \rightarrow H^2_{\nu}(\mathbb{U}^{p|q} ).
\end{align*}
\end{thm}
The proof is similar to that of theorem \ref{TCS-QE}.

\section{Toeplitz Operators}
\label{sec:Toeplitz}

\begin{defn}
For $F$ an element of $\mathcal{C}(\mathbb{B}^{p|q })$  (an element of $\mathcal{C}(\mathbb{U}^{p|q })$), the super-Toeplitz operator $T_F^{\nu}$  on $H^2_{\nu}(\mathbb{B}^{p|q })$ (on $H^2_{\nu}(\mathbb{U}^{p|q })$, respectively) is defined by
                                        $$ T_F^{\nu}\Psi = P^{\nu} (F \Psi),$$
where $P^{\nu}$ denotes the orthogonal projection onto $H^2_{\nu}(\mathbb{B}^{p|q })$ (onto $H^2_{\nu}(\mathbb{U}^{p|q })$, respectively).
\end{defn}

In \cite{LU} the authors proved that every  super
Toeplitz operator $T_F^{\nu}$ on $H^2_{\nu}(\mathbb{B}^{p|q })$ is given by the $2^q \times 2^q$ -matrix. With respect to the decomposition $$\Psi=\sum_{M\subset Q} \psi_M \xi_M H^2_{\nu}(\mathbb{B}^{p|q }),$$
the super Toeplitz operator has the form
\begin{align}\label{Toeplitz-Ball-1}
(T_F^{\nu})_{I,J}=&  \sum_{I\cup J\subset K\subset Q} \left[ \varepsilon_{K \setminus I,I}  \varepsilon_{K \setminus J,J}
\frac{\Gamma(\nu+|I|-p)}{\Gamma(\nu+|J|-p)}\right.\\
 &\times \left. T_{\nu+|I|}^{\nu+|J|}\left( F_{K \setminus I,K \setminus J} (w)  (1 -ww)^{|K|-|I|} \right)\right].
\end{align}
Here, for $0 \leq i, j \leq q$, $T_{\nu+i}^{\nu+j}$ denotes the Bergman-type Toeplitz operator
$$ T_{\nu+i}^{\nu+j}(f)= P_{\nu+i} f P_{\nu+j}: H^2_{\nu+j}(\mathbb{B}^{p })\rightarrow H^2_{\nu+i}(\mathbb{B}^{p})$$
from $H^2_{\nu+j}(\mathbb{B}^{p })$  to $H^2_{\nu+i}(\mathbb{B}^{p})$.

The following proposition is a consequence of the above result and
the fact that $H^2_{\nu}(\mathbb{B}^{p|q })$ is isometric to $H^2_{\nu}(\mathbb{U}^{p|q })$,
which is given by Theorem \ref{unitary-B-S}.

\begin{prop}
 With respect to the decomposition
 $$\Psi=\sum_{M\subset Q} \psi_M \xi_M H^2_{\nu}(\mathbb{U}^{p|q }),$$
  the super Toeplitz operator $T_F^{\nu}$ on $H^2_{\nu}(\mathbb{U}^{p|q })$ is given by the $2^q \times 2^q$ -matrix
\begin{align*}
(T_F^{\nu})_{I,J}=&  \sum_{I\cup J\subset K\subset Q} \left[ \varepsilon_{K \setminus I,I}  \varepsilon_{K \setminus J,J}
\frac{\Gamma(\nu+|I|-p)}{\Gamma(\nu+|J|-p)}\right.\\
 &\times  \left.T_{\nu+|I|}^{\nu+|J|}\left( F_{K \setminus I,K \setminus J} (w)  (\frac{ w_p-\bar{w_p}}{2i}-w'\bar{w'})^{|K|-|I|} \right)\right].
\end{align*}
Here, for $0 \leq i, j \leq q$, $T_{\nu+i}^{\nu+j}$ denotes the Bergman-type Toeplitz operator
$$ T_{\nu+i}^{\nu+j}(f)= P_{\nu+i} f P_{\nu+j}: H^2_{\nu+j}(\mathbb{U}^{p })\rightarrow H^2_{\nu+i}(\mathbb{U}^{p})$$
from $H^2_{\nu+j}(\mathbb{U}^{p })$  to $H^2_{\nu+i}(\mathbb{U}^{p})$.
\end{prop}

\begin{cor}\label{Toeplitz-Super-Cor}
Let $F$ be a super function of the form  $$F(z,\xi)= \sum_{M\subset Q} F_M(z) \xi_M \xi_M^*.$$ Then the super
Toeplitz operator $T_F^{\nu}$ on $H^2_{\nu}(\mathbb{B}^{p|q })$ (or $H^2_{\nu}(\mathbb{B}^{p|q })$)
is given by the $2^q \times 2^q$ diagonal-matrix
\begin{equation}\label{Toeplitz-Ball-2}
(T_F^{\nu})_{I,I}=  \sum_{I \subset K\subset Q}
 T_{\nu+|I|}^{\nu+|I|}\left( F_{K \setminus I} (w)  (1 -ww)^{|K|-|I|} \right),
\end{equation}
or
$$(T_F^{\nu})_{I,I}=  \sum_{I\subset K\subset Q}
 T_{\nu+|I|}^{\nu+|I|}\left( F_{K \setminus I} (w)  (\frac{ w_p-\bar{w_p}}{2i}-w'\bar{w'})^{|K|-|I|} \right).$$
\end{cor}

\section{Super Toeplitz operators with special symbols}
\label{sec:commToeplitz}

\subsection{Quasi-elliptic}

We will call  a super function $F$ quasi-elliptic if it is invariant under the action of the  quasi-elliptic group, in other words, when $F$ has the following form
$$F(z,\xi)=\sum_{I\subset Q} F_I(r)\xi_I \xi_I^*,$$
where $r = (r_1 , ..., r_p ) = (|z_1 |, \ldots, |z_p |)$.

\begin{thm}
\label{thm:q-elliptic}
 Let $F$ be a bounded measurable quasi-elliptic super function. Then
the Toeplitz operator $T_F^{\nu}$ acting on $H^2_{\nu}(\mathbb{B}^{p|q })$ is unitarily equivalent to the multiplication
operator $\gamma_{F,\nu} I = R_{\nu,(p|q)} T_F^{\nu} R_{\nu,(p|q)}^*$ acting on $(l_2(\mathbb{Z}_+^p))^{2^q}$, where $R_{\nu,(p|q)}$ and $R_{\nu,(p|q)}^*$
 are given in Definition \ref{SBT-QE}. The sequence
$$\gamma_{F,\nu}= \left(\{\gamma_{F,\nu}(n,M)\}_{n\in\mathbb{Z}^p_+}\right)_{M\subset Q}$$
is given by
\begin{align*}
 \gamma_{F,\nu}(n,M)=  \frac{ 2^p \Gamma(|n|+\nu+|M|)}{n! \Gamma(\nu+|M|-p)}    \sum_{ K\in I_M}
\int_{\tau( \mathbb{B}^p)}   F_{K \setminus M} (r)    r^{2n} (1-r^2)^{\nu+|K|-p-1} r dr,
\end{align*}
where $\mathcal{I}_M=\{K\subset Q: M\subset K\}$.
\end{thm}

\begin{proof}
For every super function $F$ the Toeplitz operator is unitarily equivalent to
$$R_{\nu,(p|q)} T_F^{\nu} R_{\nu,(p|q)}^*: (l_2(\mathbb{Z}_+^p))^{2^q}\rightarrow (l_2(\mathbb{Z}_+^p))^{2^q},$$
where the components of the operator are
\begin{equation}\label{STE-QE}
 (R_{\nu,(p|q)} T_F^{\nu} R_{\nu,(p|q)}^*)_{I,J}=  R_{\nu+|I|} (T_F^{\nu})_{I,J} R^*_{\nu+|J|}.
\end{equation}

Now, we consider a super function of the form
$F=\sum_{I\subset Q} F_I(r)\xi_I \xi_I^*$ and using corollary
 \ref{Toeplitz-Super-Cor} we have that the Toeplitz operator has the form
\begin{equation}
 (T_F^{\nu})_{I,J}=\left\{ \begin{array}{cc}
                     0 & I\neq J\\
\sum_{ K\in \mathcal{I}_I}
 T_{\nu+|I|}^{\nu+|I|}\left( F_{K \setminus I} (r)    (1 -|r|^2)^{|K|-|I|}  \right)& I=J
                    \end{array}  \right.
 \end{equation}

From the above equation and \ref{STE-QE} we have

\begin{align*}
 & (R_{\nu,(p|q)} T_F^{\nu} R_{\nu,(p|q)}^*)_{I,J}\\
 &= \left\{\begin{array}{cc}
                     0 & I\neq J\\
\sum_{K\in \mathcal{I}_I}
 R_{\nu+|I|} P_{\nu+|I|}   \left( F_{K \setminus I} (r)    (1 -|r|^2)^{|K|-|I|}\right) P_{\nu+|I|} R^*_{\nu+|I|}  & I=J
                    \end{array}  \right. .
 \end{align*}
By theorem 10.1 in \cite{QV2} the above operators are  multiplication operators then
\begin{align*}
 & \sum_{K\in \mathcal{I}_I} R_{\nu+|I|} P_{\nu+|I|}   \left( F_{K \setminus I} (r)    (1 -|r|^2)^{|K|-|I|}\right) P_{\nu+|I|} R^*_{\nu+|I|}\left(\left\{ c_{n,I}  \right\}_{n\in\mathbb{Z}_+^p}\right)\\
=& \left\{ \left[  l_{\nu+|I|}(n)  \sum_{K\in \mathcal{I}_I}
\int_{\tau( \mathbb{B}^p)} \left(  F_{K \setminus I} (r)  (1 -r^2)^{|K|-|I|} \right) \right.\right.\\
 &\times\left. \left. r^{2n} (1-r^2)^{\nu+|I|-p-1} r dr \right]c_{n,I}  \right\}_{n\in\mathbb{Z}_+^p}\\
=& \left\{ \left[ l_{\nu+|I|}(n)  \sum_{K\in \mathcal{I}_I}
\int_{\tau( \mathbb{B}^p)}  F_{K \setminus I} (r)   r^{2n} (1-r^2)^{\nu+|K|-p-1} r dr \right]c_{n,I}  \right\}_{n\in\mathbb{Z}_+^p},
\end{align*}
where
$$l_{\nu+|I|}(n)=\frac{ 2^p \Gamma(|n|+\nu+|I|)}{n! \Gamma(\nu+|I|-p)}.$$

\end{proof}

\subsection{Quasi-parabolic case}

We will call  a super function $F$ quasi-parabolic if it is invariant under the action of the  quasi-parabolic group, in other words, when $F$
has the following form
$$F(z,\xi)=\sum_{I\subset Q} f_I(r, \text{Im }(z_n))\xi_I \xi_I^*,$$
where $r = (r_1 , ..., r_{p-1} ) = (|z_1 |, \ldots, |z_{p-1} |)$.

\begin{thm}
\label{thm:q-parabolic}
 Let $F$ be a bounded measurable quasi-parabolic super function. Then
the Toeplitz operator $T_F^{\nu}$ acting on $H^2_{\nu}(\mathbb{D}^{p|q })$ is unitarily equivalent to the multiplication
operator $\gamma_{F,\nu} I = R_{\nu,(p|q)} T_F^{\nu} R_{\nu,(p|q)}^*$ acting on \\ $(l_2(\mathbb{Z}_+^{p-1}, L_2(\mathbb{R}_+)))^{2^q}$, where $R_{\nu,(p|q)}$ and $R_{\nu,(p|q)}^*$
 are given in Definition \ref{SBT-QP}. The sequence
$$\gamma_{F,\nu}= \left(\{\gamma_{F,\nu}(n,\xi,M)\}_{n\in\mathbb{Z}^{p-1}_+}\right)_{M\subset Q}$$ is given by
\begin{align*}
\gamma_{F,\nu}(n,\xi,M)=&
     \frac{ (2\xi)^{|n|+\nu+|M|-1}}{n! \Gamma(\nu+|M|-p)}  \\
     &\times  \sum_{K\in \mathcal{I}_M}
 \int_{ \mathbb{R}^p_+}
  F_{K \setminus M} (\sqrt{r'},v+\hat{r})
r^{n} e^{-2\xi(v+\hat{r})} v^{\nu+|K|-p-1} drdv,
\end{align*}
where $r=(r_1,\ldots,r_{p-1})$, $\sqrt{r}=(\sqrt{r}_1,\ldots,\sqrt{r}_{p-1})$, $\hat{r}=r_1+\cdots+r_{p-1}$ and $\mathcal{I}_M=\{K\subset Q: M\subset K\}$.
\end{thm}

\begin{proof}
For every super function $F$ the Toeplitz operator is unitarily equivalent to
$$R_{\nu,(p|q)} T_F^{\nu} R_{\nu,(p|q)}^*: (l_2(\mathbb{Z}_+^{p-1}, L_2(\mathbb{R}_+)))^{2^q}\rightarrow (l_2(\mathbb{Z}_+^{p-1}, L_2(\mathbb{R}_+)))^{2^q},$$
where the components of the operator are
\begin{equation}\label{Toeplitz-equivalent-parabolic}
 (R_{\nu,(p|q)} T_F^{\nu} R_{\nu,(p|q)}^*)_{I,J}=  R_{\nu+|I|} (T_F^{\nu})_{I,J} R^*_{\nu+|J|}.
\end{equation}

Now, we consider a super function
$F=\sum_{I\subset Q} F_I(r,  \text{Im}(z_n)  )\xi_I \xi_I^*$
 and using corollary
 \ref{Toeplitz-Super-Cor} we obtain
\begin{align*}
 (T_F^{\nu})_{I,J}=\left\{ \begin{array}{cc}
                     0 & I\neq J\\
\sum_{K\in \mathcal{I}_I}
 T_{\nu+|I|}^{\nu+|I|}\left( F_{K \setminus I} (r,v+\hat{r})    v^{|K|-|I|}  \right)& I=J
                    \end{array}  \right..
 \end{align*}

From the above equation and \ref{Toeplitz-equivalent-parabolic} we have

\begin{align*}
 & (R_{\nu,(p|q)} T_F^{\nu} R_{\nu,(p|q)}^*)_{I,J}\\
 &= \left\{\begin{array}{cc}
                     0 & I\neq J\\
\sum_{K\in \mathcal{I}_I}
 R_{\nu+|I|} P_{\nu+|I|}   \left( F_{K \setminus I} (r,v+\hat{r})    (v)^{|K|-|I|}\right) P_{\nu+|I|} R^*_{\nu+|I|}  & I=J
                    \end{array}  \right..
 \end{align*}
By Theorem 10.2 in \cite{QV2} the above operators are multiplication operators
\begin{align*}
 & \sum_{K\in \mathcal{I}_I} R_{\nu+|I|} P_{\nu+|I|}   \left(F_{K \setminus I} (r',v+\hat{r})    (v)^{|K|-|I|}\right)\\
  & \times P_{\nu+|I|} R^*_{\nu+|I|}\left(\left\{ c_{n,I}(\xi)  \right\}_{n\in\mathbb{Z}_+^{p-1}}\right)\\
=& \left\{
\left[ \sum_{K\in \mathcal{I}_I} \frac{ (2\xi)^{|n|+\nu+|I|-1}}{n! \Gamma(\nu+|M|-p)}
 \int_{ \mathbb{R}^p_+}
  \left(F_{K \setminus I} (\sqrt{r'},v+\hat{r})v^{|K|-|I|}\right)
  \right.\right.\\
& \times\left.\left.
r^{n} e^{-2\xi(v+\hat{r})} v^{\nu+|I|-p-1} drdv \right]
c_{n,I}(\xi)  \right\}_{n\in\mathbb{Z}_+^{p-1}}\\
=& \left\{ \left[ \frac{ (2\xi)^{|n|+\nu+|I|-1}}{n! \Gamma(\nu+|M|-p)}   \sum_{K\in \mathcal{I}_I}
 \int_{ \mathbb{R}^p_+}
  F_{K \setminus I} (\sqrt{r'},v+\hat{r})
  \right.\right.\\
& \times\left.\left.
r^{n} e^{-2\xi(v+\hat{r})} v^{\nu+|K|-p-1} drdv\right]c_{n,I}(\xi)  \right\}_{n\in\mathbb{Z}_+^{p-1}}.
\end{align*}

\end{proof}

\subsection{Nilpotent case}

We will call  a super function $F$ nilpotent if it is invariant under the action of the  Nilpotent group, in other words, when $F$
has the following form
$$F(z,\xi)=\sum_{I\subset Q} f_I(\Ima z',\Ima z_n - |z'|^2) \xi_I \xi_I^*,$$
where $ z'=(z_1 , ..., z_{p-1} )$.

\begin{thm}
\label{thm:nilpotent}
 Let $F$ be a bounded measurable nilpotent super function. Then
the Toeplitz operator $T_F^{\nu}$ acting on $H^2_{\nu}(\mathbb{U}^{p|q })$ is unitarily equivalent to the multiplication
operator $\gamma_{F,\nu} I = R_{\nu,(p|q)} T_F^{\nu} R_{\nu,(p|q)}^*$ acting on $(L_2(\R^{p-1} \times \R_+))^{2^q}$, where $R_{\nu,(p|q)}$ and $R_{\nu,(p|q)}^*$
 are given in Definition \ref{SBT-N}. The function
$$\gamma_{F,\nu}= \left( \gamma_{F,\nu}(u',\xi,M)   \right)_{M\subset Q}$$ is given by
\begin{align*}
\gamma_{F,\nu}(u',\xi,M)=&
     \frac{(2\xi)^{\nu+|M|-p}} {\pi^{\frac{p-1}{2}}\Gamma(\nu+|M|-p)}
     \sum_{K\in \mathcal{I}_M}
\int_{\R^{n-1} \times \R_+}  \\
 &\times F_{K \setminus M}(\frac{1}{2\sqrt{\xi}}(-u'+v'),v) e^{-2\xi v -|v'|^2}\,v^{\nu+|K|-p-1}dv'dv,
\end{align*}
where $\mathcal{I}_M=\{K\subset Q: M\subset K\}$,  $u' \in \R^{n-1}$ and $\xi \in \R_+$.
\end{thm}

\begin{proof}
For every super function $F$ the Toeplitz operator is unitarily equivalent to
$$R_{\nu,(p|q)} T_F^{\nu} R_{\nu,(p|q)}^*:(L_2(\R^{p-1} \times \R_+))^{2^q}\rightarrow (L_2(\R^{p-1} \times \R_+))^{2^q},$$
where the components of the operator are
\begin{equation}\label{Toeplitz-equivalent-nil}
 (R_{\nu,(p|q)} T_F^{\nu} R_{\nu,(p|q)}^*)_{I,J}=  R_{\nu+|I|} (T_F^{\nu})_{I,J} R^*_{\nu+|J|}.
\end{equation}

Now, we consider a super function of the form
$$F=\sum_{I\subset Q} F_I(\Ima z',\Ima z_n - |z'|^2) \xi_I \xi_I^*.$$
Using \refeq{Toeplitz-Super-Cor} we obtain
\begin{equation}
 (T_F^{\nu})_{I,J}=\left\{ \begin{array}{cc}
                     0 & I\neq J\\
\sum_{K\in \mathcal{I}_I}
 T_{\nu+|I|}^{\nu+|I|}\left( F_{K \setminus I} (\Ima z',\Ima z_n - |z'|^2)  v^{|K|-|I|}  \right)& I=J
                    \end{array}  \right..
 \end{equation}

From the above equation and \ref{Toeplitz-equivalent-nil} we have

\begin{align*}
&  (R_{\nu,(p|q)} T_F^{\nu} R_{\nu,(p|q)}^*)_{I,J}\\
 &= \left\{\begin{array}{cc}
                     0 & I\neq J\\
\sum_{K\in \mathcal{I}_I}
 R_{\nu+|I|} P_{\nu+|I|}   \left( F_{K \setminus I} (y',v)    (v)^{|K|-|I|}\right) P_{\nu+|I|} R^*_{\nu+|I|}  & I=J
                    \end{array}  \right..
 \end{align*}
By theorem 10.3 in \cite{QV2} the above operators are the multiplication operators
\begin{align*}
  &\sum_{K\in \mathcal{I}_I} R_{\nu+|I|} P_{\nu+|I|}   \left(F_{K \setminus I} (y',v)    (v)^{|K|-|I|}\right) P_{\nu+|I|} R^*_{\nu+|I|} \left(  c_{I}(u', \xi)  \right)\\
=&
\left[ \sum_{K\in \mathcal{I}_I}
\frac{(2\xi)^{\nu+|M|-p}} {\pi^{\frac{p-1}{2}}\Gamma(\nu+|M|-p)}
\int_{\R^{n-1} \times \R_+}  F_{K \setminus M}(\frac{1}{2\sqrt{\xi}}(-u'+v'),v)\right.\\
&\times \left.
 e^{-2\xi v -|v'|^2}\,v^{\nu+|K|-p-1}dv'dv
 \right]
c_{I}(u',\xi) .
\end{align*}

\end{proof}

\subsection{Quasi-nilpotent case}

We will call  a super function $F$ quasi-nilpotent if it is invariant under the action of the  quasi-nilpotent group, in other words, when $F$
has the following form
$$F(z,\xi)=\sum_{I\subset Q} f_I(r, y',\Ima z_n - |z'|^2) \xi_I \xi_I^*,$$
where   $r=(|z_1|,...,|z_k|)$, $y' = \Ima w'$, and $w'=(z_{k+1},...,z_{p-1})$.

\begin{thm}
\label{thm:q-nilpotent}
 Let $F$ be a bounded measurable quasi-nilpotent super function. Then
the Toeplitz operator $T_F^{\nu}$ acting on $H^2_{\nu}(\mathbb{U}^{p|q })$ is unitarily equivalent to the multiplication
operator $\gamma_{F,\nu} I = R_{\nu,(p|q)} T_F^{\nu} R_{\nu,(p|q)}^*$ acting on $(l_2(\Z^k_+,L_2(\R^{p-k-1} \times \R_+)))^{2^q}$, where $R_{\nu,(p|q)}$ and $R_{\nu,(p|q)}^*$
 are given in Definition \ref{SBT-QN}. The function
$$\gamma_{F,\nu}= \left( \{ \gamma_{F,\nu}(n,u',\xi,M)\}_{n \in \Z_+^k}   \right)_{M\subset Q}$$ is given by
\begin{align*}
 \gamma_{F,\nu}(n,u',\xi,M)
 =& \pi^{-\frac{p-k-1}{2}}\,\frac{(2\xi)^{|n|+\nu+|M|-p +k}}{n!\, \Gamma(\nu+|M|-p)}
     \sum_{K\in \mathcal{I}_M}     \int_{\R^k_+\times \R^{n-k-1} \times \R_+}\\
     &\times F_{K \setminus M}(\sqrt{r},\frac{1}{2\sqrt{\xi}}(-u'+v'),v+r_1+...+r_k)
 r^p\\
 &\times  e^{-2\xi(v+r_1+...+r_k)-|v'|^2}\,v^{\nu+|K|-p-1}drdv'dv,
\end{align*}
where $\mathcal{I}_M=\{K\subset Q: M\subset K\}$,  $\sqrt{r}=(\sqrt{r}_1,...,\sqrt{r}_k)$, $n \in \Z_+^k$, $u' \in \R^{n-k-1}$ and $\xi \in \R_+$.
\end{thm}

\begin{proof}
For every super function $F$ the Toeplitz operator is unitarily equivalent to
$$R_{\nu,(p|q)} T_F^{\nu} R_{\nu,(p|q)}^*: (l_2(\Z^k_+,L_2(\R^{p-k-1} \times \R_+)))^{2^q}\rightarrow (l_2(\Z^k_+,L_2(\R^{p-k-1} \times \R_+)))^{2^q},$$
where the components of the operator are
\begin{equation}\label{Toeplitz-equivalent-qnil}
 (R_{\nu,(p|q)} T_F^{\nu} R_{\nu,(p|q)}^*)_{I,J}=  R_{\nu+|I|} (T_F^{\nu})_{I,J} R^*_{\nu+|J|}.
\end{equation}

Now, we consider a super function of the form
$F=\sum_{I\subset Q} F_I(r, y',\Ima z_n - |z'|^2) \xi_I \xi_I^*$ and using cororally
 \ref{Toeplitz-Super-Cor} we obtain
\begin{equation}
 (T_F^{\nu})_{I,J}=\left\{ \begin{array}{cc}
                     0 & I\neq J\\
\sum_{K\in \mathcal{I}_I}
 T_{\nu+|I|}^{\nu+|I|}\left( F_{K \setminus I} (r, y',\Ima z_n - |z'|^2)  v^{|K|-|I|}  \right)& I=J
                    \end{array}  \right..
 \end{equation}

From the above equation and \ref{Toeplitz-equivalent-qnil} we have

\begin{align*}
 & (R_{\nu,(p|q)} T_F^{\nu} R_{\nu,(p|q)}^*)_{I,J}\\
 &= \left\{\begin{array}{cc}
                     0 & I\neq J\\
\sum_{K\in \mathcal{I}_I}
 R_{\nu+|I|} P_{\nu+|I|}   \left( F_{K \setminus I} (r, y',v)    (v)^{|K|-|I|}\right) P_{\nu+|I|} R^*_{\nu+|I|}  & I=J
                    \end{array}  \right..
 \end{align*}
In particular by theorem 10.4 in \cite{QV2} the following operator is a multiplication operator
\begin{align*}
 & \sum_{K\in \mathcal{I}_I} R_{\nu+|I|} P_{\nu+|I|}   \left(F_{K \setminus I} (r, y',v)    (v)^{|K|-|I|}\right) P_{\nu+|I|} R^*_{\nu+|I|} \left(   c_{n,I}(u', \xi)   \right)\\
=&
 \sum_{K\in \mathcal{I}_I}
\pi^{-\frac{p-k-1}{2}}\,\frac{(2\xi)^{|n|+\nu+|M|-p +k}}{n!\, \Gamma(\nu+|M|-p)}
        \int_{\R^k_+\times \R^{n-k-1} \times \R_+}\\
 &\times  F_{K \setminus M}(\sqrt{r},\frac{1}{2\sqrt{\xi}}(-u'+v'),v+r_1+...+r_k)\\
&\times  r^n\, e^{-2\xi(v+r_1+...+r_k)-|v'|^2}\,v^{\nu+|K|-p-1}drdv'dv
  \cdot [c_{n,I}(u', \xi)].
\end{align*}

\end{proof}

\subsection{Quasi-hyperbolic case}

We will call  a super function $F$ quasi-nilpotent if it is invariant under the action of the  Quasi-nilpotent group, in other words, when $F$
has the following form
$$F(z,\xi)=\sum_{I\subset Q} f_I\left(s_1, ...,s_{n-1}, \, \arg(z_n - i|z'|^2)\right) |z_n - i|z'|^2|^{-|I|}  \xi_I \xi_I^*,$$
where
$$s_k= \frac{|z_k|}{\sqrt{|z'|^2 + |z_n - i|z'|^2|}}, $$
$k=1,\ldots,n-1$ and $z'=(z_{1},...,z_{p-1})$.

\begin{thm}
\label{thm:q-hyperbolic}
 Let $F$ be a bounded measurable quasi-nilpotent super function. Then
the Toeplitz operator $T_F^{\nu}$ acting on $H^2_{\nu}(\mathbb{U}^{p|q })$ is unitarily equivalent to the multiplication
operator $\gamma_{F,\nu} I = R_{\nu,(p|q)} T_F^{\nu} R_{\nu,(p|q)}^*$ acting on $(l_2(\Z^{p-1}_+,L_2(\R)))^{2^q}$, where $R_{\nu,(p|q)}$ and $R_{\nu,(p|q)}^*$
 are given in Definition \ref{SBT-QH}. The function
$$\gamma_{F,\nu}= \left( \{ \gamma_{F,\nu}(n,\xi,M)\}_{n \in \Z_+^{p-1}}   \right)_{M\subset Q}$$ is given by
\begin{align*}
 \gamma_{F,\nu}(n,\xi,M) =&
\sum_{K\in \mathcal{I}_M} \alpha^2_{n,\nu+|M|}(\xi) \int_{\tau(\B^{n-1}) \times (0,\pi)}   F_{K \setminus M}(s,\theta) \, |\beta_{n,\nu+|M|}(s,\xi,\theta)|^2 \\
&\times
(1-|s|^2)^{\nu+|M|-p} \frac{c_{\nu+|M|}}{4}\,\sin^{\nu+|K|-p-1} \theta\, sdsd\theta,
\end{align*}
where $\mathcal{I}_M=\{K\subset Q: M\subset K\}$, and the functions $\alpha_p(\xi)$ and $\beta_p(s,\xi,\theta)$ are given by (\ref{eq:alpha_p}) and (\ref{eq:beta_p}), respectively.
\end{thm}

\begin{proof}
For every super function $F$ the Toeplitz operator is unitarily equivalent to
$$R_{\nu,(p|q)} T_F^{\nu} R_{\nu,(p|q)}^*: (l_2(\Z^{p-1}_+,L_2(\R)))^{2^q}\rightarrow (l_2(\Z^{p-1}_+,L_2(\R)))^{2^q},$$
where the components of the operator are
\begin{equation}\label{Toeplitz-equivalent-h}
 (R_{\nu,(p|q)} T_F^{\nu} R_{\nu,(p|q)}^*)_{I,J}=  R_{\nu+|I|} (T_F^{\nu})_{I,J} R^*_{\nu+|J|}.
\end{equation}

Now, we consider a super function of the form
$F=\sum_{I\subset Q} F_I (s,\theta)r^{-|I|}  \xi_I \xi_I^*$ and using
 corollary \ref{Toeplitz-Super-Cor} we obtain
\begin{equation}
 (T_F^{\nu})_{I,J}=\left\{ \begin{array}{cc}
                     0 & I\neq J\\
\sum_{K\in \mathcal{I}_I}
 T_{\nu+|I|}^{\nu+|I|}\left( F_{K \setminus I} (s,\theta)r^{-|K \setminus I|}  v^{|K|-|I|}  \right)& I=J
                    \end{array}  \right..
 \end{equation}

From the above equation and \ref{Toeplitz-equivalent-h} we have

\begin{align*}
& (R_{\nu,(p|q)} T_F^{\nu} R_{\nu,(p|q)}^*)_{I,J}\\
&= \left\{\begin{array}{cc}
                     0 & I\neq J\\
\sum_{K\in \mathcal{I}_I}
 R_{\nu+|I|} P_{\nu+|I|}   \left( F_{K \setminus I} (s,\theta)  \sin^{|K|-|I|}\theta\right) P_{\nu+|I|} R^*_{\nu+|I|}  & I=J
                    \end{array}  \right..
 \end{align*}
By theorem 10.5 in \cite{QV2} the above operators are multiplication operators
\begin{align*}
 & \sum_{K\in \mathcal{I}_I} R_{\nu+|I|} P_{\nu+|I|}  F_{K \setminus I} (s,\theta)  \sin^{|K|-|I|} \theta     P_{\nu+|I|} R^*_{\nu+|I|} \left(   c_{n,I}( \xi)   \right)\\
=& \sum_{K\in \mathcal{I}_I}
  \alpha^2_{n,\nu+|I|}(\xi) \int_{\tau(\B^{n-1}) \times (0,\pi)}   F_{K \setminus I}(s,\theta) \sin^{|K|-|I|}\theta  \, |\beta_{n,\nu+|I|}(s,\xi,\theta)|^2
 \\ & \times(1-|s|^2)^{\nu+|I|-p}\frac{c_{\nu+|I|}}{4}\,\sin^{\nu+|I|-p-1} \theta\, sdsd\theta
  \cdot [c_{n,I}(u', \xi)]   \\
  =& \sum_{K\in \mathcal{I}_I}
  \alpha^2_{n,\nu+|I|}(\xi) \int_{\tau(\B^{n-1}) \times (0,\pi)}   F_{K \setminus I}(s,\theta)  \, |\beta_{n,\nu+|I|}(s,\xi,\theta)|^2 \,
 \\ & \times(1-|s|^2)^{\nu+|I|-p}\frac{c_{\nu+|I|}}{4}\,\sin^{\nu+|K|-p-1} \theta\, sdsd\theta
  \cdot [c_{n,I}(u', \xi)].
\end{align*}

\end{proof}



\end{document}